\documentclass[a4paper,11pt,reqno]{amsart}
\usepackage[utf8]{inputenc}
\usepackage{amsmath}
\usepackage{amssymb}
\usepackage{amsthm}
\usepackage{mathtools}
\usepackage{hyperref}
\usepackage{color}
\usepackage{enumitem}
\usepackage{a4wide}
\usepackage{mathdots}

\DeclareMathOperator{\GL}{GL}
\DeclareMathOperator{\gl}{\mathfrak{gl}}
\DeclareMathOperator{\PGL}{PGL}
\DeclareMathOperator{\SL}{SL}

\DeclareMathOperator{\upO}{O}

\newcommand{\fraka}{\mathfrak{a}}

\newcommand{\frakg}{\mathfrak{g}}
\newcommand{\frakh}{\mathfrak{h}}

\newcommand{\frakn}{\mathfrak{n}}

\newcommand{\CC}{\mathbb{C}}

\newcommand{\PP}{\mathbb{P}}
\newcommand{\RR}{\mathbb{R}}
\newcommand{\ZZ}{\mathbb{Z}}

\newcommand{\calD}{\mathcal{D}}

\newcommand{\calW}{\mathcal{W}}

\newcommand{\0}{\textbf{0}}
\newcommand{\1}{{\rm\bf 1}}

\DeclareMathOperator{\Ind}{Ind}

\DeclareMathOperator{\Hom}{Hom}

\DeclareMathOperator{\diag}{diag}

\renewcommand\Re{\operatorname{Re}}

\renewcommand{\mod}{{\rm mod}}
\newcommand{\aut}{{\rm aut}}

\makeatletter
\newcommand{\subalign}[1]{
  \vcenter{
    \Let@ \restore@math@cr \default@tag
    \baselineskip\fontdimen10 \scriptfont\tw@
    \advance\baselineskip\fontdimen12 \scriptfont\tw@
    \lineskip\thr@@\fontdimen8 \scriptfont\thr@@
    \lineskiplimit\lineskip
    \ialign{\hfil$\m@th\scriptstyle##$&$\m@th\scriptstyle{}##$\crcr
      #1\crcr
    }
  }
}
\makeatother

\theoremstyle{plain}
\newtheorem{theorem}{Theorem}[section]
\newtheorem*{theorem*}{Theorem}
\newtheorem{proposition}[theorem]{Proposition}
\newtheorem{lemma}[theorem]{Lemma}
\newtheorem{corollary}[theorem]{Corollary}
\newtheorem{conjecture}[theorem]{Conjecture}

\newtheorem*{fact*}{Fact}

\theoremstyle{definition}

\newtheorem{remark}[theorem]{Remark}

\numberwithin{equation}{section}

\title[Rankin--Selberg periods for spherical principal series]{Rankin--Selberg periods for spherical principal series}

\author{Jan Frahm}
\address{Department of Mathematics, Aarhus University, Ny Munkegade 118, 8000 Aarhus C, Denmark}
\email{frahm@math.au.dk}

\author{Feng Su}
\address{Department of Pure Mathematics, Xi'an Jiaotong--Liverpool University, 111 Ren'ai Road, Suzhou Industrial Park, Suzhou 215123, China}
\email{feng.su@xjtlu.edu.cn}

\begin{document}



\maketitle

\begin{abstract}
By the unfolding method, Rankin--Selberg $L$-functions for $\GL(n)\times\GL(n')$ can be expressed in terms of period integrals. These period integrals actually define invariant forms on tensor products of the relevant automorphic representations. By the multiplicity-one theorems due to Sun--Zhu and Chen--Sun such invariant forms are unique up to scalar multiples and can therefore be related to invariant forms on equivalent principal series representations. We construct meromorphic families of such invariant forms for spherical principal series representations of $\GL(n,\RR)$ and conjecture that their special values at the spherical vectors agree in absolute value with the archimedean local $L$-factors of the corresponding $L$-functions. We verify this conjecture in several cases.\\
This work can be viewed as the first of two steps in a technique due to Bernstein--Reznikov for estimating $L$-functions using their period integral expressions.
\end{abstract}

\section*{Introduction}

To a pair of automorphic forms $f$ on $\GL(n)$ and $g$ on $\GL(n')$ one can associate the Rankin--Selberg $L$-function $L(s,f\times\overline{g})$, which is given by a Dirichlet series involving the Fourier--Whittaker coefficients of $f$ and $g$. Rankin--Selberg $L$-functions generalize the standard Godement--Jacquet $L$-function and are holomorphic/meromorphic functions of $s\in\CC$ satisfying an explicit functional equation. One of the fundamental problems concerning the analytic aspects of $L$-functions is to bound $L(s,f\times\overline{g})$ along the critical line $\frac{1}{2}+i\RR$ in terms of $s$ and/or the Langlands parameters of $f$ and $g$. Using the functional equation and the Phragm\'{e}n--Lindel\"{o}f convexity principle one obtains the so-called \emph{convexity bound} for $L(s,f\times\overline{g})$ on the critical line. Any bound improving the convexity bound is called a \emph{subconvexity bound}.

Subconvexity bounds for Rankin--Selberg $L$-functions have been obtained in various aspects (see e.g. \cite{Blo12,BBM17,Li11,LY02,MSY18,Mun15,You11,You13} and references therein), but so far mostly for $n,n'\leq3$. The methods used involve deep techniques from analytic number theory such as trace formulas, the circle method or the delta method, and often require delicate analytical tools such as stationary phase approximation. After this paper was finished, certain subconvexity bounds for the case $n'=n-1$ were obtained by P.~Nelson~\cite{Nel20} using different methods.

There is another way of approaching Rankin--Selberg $L$-functions, that is through period integrals. Integrating the automorphic forms over a certain locally symmetric space, involving an Eisenstein series in the case $n'=n$, defines a period integral $\Lambda(s,f\times\overline{g})$. By the unfolding method, this period integral $\Lambda(s,f\times\overline{g})$ essentially reduces to the product of an archimedean local $L$-factor $G_{\lambda,\nu}(s)$, which only depends on the Langlands parameters $\lambda$ of $f$ and $\nu$ of $\overline{g}$, and the Rankin--Selberg $L$-function $L(s,f\times\overline{g})$. This makes it possible to bound Rankin--Selberg $L$-functions by bounding the corresponding period integrals $\Lambda(s,f\times\overline{g})$.

In a series of seminal papers \cite{BR99,BR04,BR05,BR10,Rez08}, Bernstein and Reznikov were able to obtain subconvexity bounds for triple product $L$-functions of $\PGL(2)$ by bounding the corresponding period integrals. Their method involves representation theory of the group $\PGL(2,\RR)$; more precisely they study in detail invariant trilinear functionals on products of the corresponding automorphic representations. The key ingredient is the multiplicity one property, which asserts that invariant trilinear functionals on products of irreducible representations of $\PGL(2,\RR)$ are unique up to scalar multiples. They proceed essentially in two steps:
\begin{enumerate}[label=(\Roman*)]
\item Explicit invariant trilinear functionals for principal series representations are constructed and evaluated at the spherical vectors. These explicit functionals are by the multiplicity one property proportional to the functionals given by the period integrals.
\item The proportionality scalar is bounded above by different methods (analytic continuation of representations, bound for $L^4$-norms of $K$-types, estimates for Hermitian forms on automorphic representations).
\end{enumerate}

In this paper we study step (I) in the framework of Rankin--Selberg $L$-functions. In the same way as for $\PGL(2)$, the period integrals $\Lambda(s,f\times\overline{g})$ related to Rankin--Selberg $L$-functions for $\GL(n)\times\GL(n')$ give rise to invariant functionals on the product of the automorphic representations corresponding to $f$ and $\overline{g}$. The multiplicity one property in this framework was established by Sun--Zhu \cite{SZ12} for $n'=n-1$ and by Chen--Sun \cite{CS15} for general $1\leq n'\leq n$ and it allows one to relate these invariant functionals to explicit invariant functionals on principal series representations. Our main results are:

\begin{itemize}
\item The construction of explicit invariant functionals on tensor products of spherical principal series representations in terms of their integral kernels (see Proposition \ref{prop:ModelFormsEqual} for the case $n'=n$ and Proposition \ref{prop:ModelFormsUnequal} for the case $1\leq n'\leq n-1$).
\item The conjecture that the special value of our invariant functionals at the spherical vectors equals in absolute value the local $L$-factor $G_{\lambda,\nu}(s)$ up to a constant (see Conjecture \ref{conj:MatchingGfunctionModelPeriod}).
\item Verification of our conjecture for $(n,n')=(2,2)$, $(2,1)$, $(3,2)$ and $(3,1)$ (see Sections \ref{sec:GL2} and \ref{sec:GL3}).
\end{itemize}

The construction and study of invariant functionals on principal series is itself an interesting topic in representation theory, which has recently received much attention in the framework of branching problems (see e.g. \cite{KS15,KS18,Moe17} and references therein). More precisely, the Gan--Gross--Prasad conjectures can be viewed as a generalization of Rankin--Selberg theory. For instance, the corresponding model periods for rank one orthogonal groups have recently been studied by Kobayashi--Speh \cite{KS19}.

To use our results in order to obtain bounds for the corresponding Rankin--Selberg $L$-functions, one needs to carry out step (II) as well. We hope to return to this in a subsequent paper. At this moment it is not clear to us, which of the different methods in \cite{BR99,BR04,BR10} to estimate the proportionality scalars generalizes to our situation. The most straightforward technique seems to be estimating Hermitian forms on automorphic representations (see e.g. \cite{BR04}), but it requires that the quotient $\Gamma\backslash G$ is compact which is not the case for $\Gamma=\SL(n,\ZZ)$.

We remark that there is a much simpler way of constructing explicit invariant functionals on tensor products of generic representations of general linear groups using Whittaker models. More precisely, a generic representation $\pi$ of $\GL(n,\RR)$ admits a realization on a subspace
$$ \calW(\pi,\psi)\subseteq\{W\in C^\infty(\GL(n,\RR)):W(ng)=\psi(n)W(g)\,\forall\,g\in\GL(n,\RR),n\in N_n\}, $$
where $\psi$ is a non-degenerate character of the subgroup $N_n\subseteq\GL(n,\RR)$ of unipotent upper triangular matrices. If $\calW(\tau,\overline{\psi})$ denotes the corresponding model for a generic representation $\tau$ of $\GL(n',\RR)$, embedded in $\GL(n,\RR)$ as the upper left corner, where $\overline{\psi}$ is the restriction of the complex conjugate character to the maximal unipotent subgroup $N_{n'}=N_n\cap\GL(n',\RR)$, then
$$ \calW(\pi,\psi)\widehat{\otimes}\calW'(\tau,\overline{\psi})\to\CC, \quad W\otimes W'\mapsto\int_{N_{n'}\backslash\GL(n',\RR)}W\begin{pmatrix}h&\\&\1_{n-n'}\end{pmatrix}W'(h)|\det(h)|^{s-\frac{n-n'}{2}}\,dh $$
defines an invariant functional in the case $n'<n$, and a similar construction can be carried out for $n=n'$. (Here $\widehat{\otimes}$ denotes the completed projective tensor product.) The method of Bernstein--Reznikov relies on an explicit description of the group action, the representation space and the invariant inner product in order to construct test vectors and estimate their invariant norms when acted upon by group elements close to the identity element. In the Whittaker model $\calW(\pi,\psi)$ the group action is the right regular representation and hence very explicit. However, the invariant inner product and the explicit description of smooth vectors in the representation space are only accessible when restricting $W\in\calW(\pi,\psi)$ to $\GL(n-1,\RR)$, also referred to as the Kirillov model. In the Kirillov model functions with compact support modulo $N_{n-1}$ are smooth vectors and the invariant inner product is given by integration over $N_{n-1}\backslash\GL(n-1,\RR)$, so both the representation space and the invariant inner product are somehow explicit. But it is non-trivial to recover $W$ from its restriction to $\GL(n-1,\RR)$ and hence find an expression for the group action. This is the reason we are interested in invariant functionals on tensor products of principal series whose representation space is explicitly given as a space of sections of a vector bundle over the flag variety, the group action being given by the right regular representation and the invariant inner product simply being an $L^2$-inner product.

\subsection*{Structure of the paper} In Section \ref{sec:RSLfcts} we recall the definition of Rankin--Selberg $L$-functions and convolutions, including known results about the archimedean local $L$-factors. Section \ref{sec:AutRSperiods} is about interpreting the period integrals as special values of invariant forms on automorphic representations. In Section \ref{sec:ModRSperiods} we construct explicit invariant forms on tensor products of principal series and conjecture their values at the spherical vectors. These values are explicitly computed for $(n,n')=(2,2)$ and $(2,1)$ in Section \ref{sec:GL2} and for $(n,n')=(3,2)$ and $(3,1)$ in Section \ref{sec:GL3}. Finally, in the Appendix \ref{app:IntFormulas} we collect the integral formulas needed in Section \ref{sec:GL2} and \ref{sec:GL3} to evaluate the relevant integrals.

\subsection*{Acknowledgments} We thank Binyong Sun for sharing his insights and ideas on invariant functionals and for bringing the paper \cite{CS15} to our attention. We are also grateful to the referee for several helpful suggestions and remarks. The first author was supported by a research grant from the Villum Foundation (Grant No.~00025373). The second author was partly supported by the National Natural Science Foundation of China (No.~11901466) and the XJTLU Research Development Funding (RDF-19-02-04).

\section{Rankin--Selberg $L$-functions and period integrals}\label{sec:RSLfcts}

We recall the necessary facts about Maass forms on $\GL(n,\RR)$, their Fourier--Whittaker expansions and Rankin--Selberg $L$-functions, including the relation between $L$-functions and period integrals provided by the unfolding method. For this we mostly follow \cite{Gol06}.

\subsection{Maass forms on $\GL(n,\RR)$}\label{sec:MaassForms}

Let $G=\GL(n,\RR)$ ($n\geq2$), $Z(G)=\RR^\times$ its center and $K_G=\upO(n)$ the standard maximal compact subgroup. We fix the lattice $\Gamma_G=\SL(n,\ZZ)$. Further, let $\frakg=\gl(n,\RR)$ denote the Lie algebra of $G$, $U(\frakg)$ the universal enveloping algebra of its complexification $\frakg_\CC$ and $Z(\frakg)$ the center of $U(\frakg)$. Then $G$ acts unitarily on $L^2(\Gamma_G\cdot Z(G)\backslash G)$ by right-translation and this induces an action of $U(\frakg)$ on $C^\infty(\Gamma_G\cdot Z(G)\backslash G)$ by differential operators.

A function $f\in C^\infty(\Gamma_G\cdot Z(G)\backslash G)$ is called a {\it Maass form} if it has the following properties:
\begin{enumerate}
\item\label{label:DefMaassForm1}
$f$ is $K_G$-invariant,
\item\label{label:DefMaassForm2}
$f$ is a joint eigenfunction of $Z(\frakg)$, i.e. $Df=\lambda_Df$,~$\forall\,D\in Z(\frakg)$ for some scalars $\lambda_D\in\CC$,
\item\label{label:DefMaassForm3}
$f\in L^2(\Gamma_G\cdot Z(G)\backslash G)$.
\end{enumerate}
A Maass form $f$ is furthermore called {\it cusp form} if it decays rapidly at the cusp of the locally symmetric space $\Gamma_G\cdot Z(G)\backslash G/K_G$.

In view of property \eqref{label:DefMaassForm1}, Maass forms can also be viewed as $\Gamma_G$-invariant smooth functions on the semisimple Riemannian symmetric space $G/K_G\cdot Z(G)$, which can be identified with the generalized upper half plane. For this let $A_G$ denote the subgroup of diagonal matrices with positive diagonal entries and $N_G$ the unipotent subgroup of upper triangular matrices. In view of the Iwasawa decomposition $G=N_GA_GK_G$, the quotient $G/K_G\cdot Z(G)=\GL(n,\RR)/\upO(n)\cdot\RR^\times$ can be identified with the \textit{generalized upper half plane} $\frakh^n$, which is defined to be the set of all products $z=xy$ with
\begin{equation}
 x = \begin{pmatrix}1&x_{1,2}&\cdots&\cdots&x_{1,n}\\&1&x_{2,2}&\cdots&x_{2,n}\\&&\ddots&&\vdots\\&&&1&x_{n-1,n}\\&&&&1\end{pmatrix}, \qquad y = \begin{pmatrix}y_1y_2\cdots y_{n-1}&&&&\\&\ddots&&&\\&&y_1y_2&&\\&&&y_1&\\&&&&1\\\end{pmatrix},\label{eq:DefinitionXY}
\end{equation}
where $x_{ij}\in\RR$ and $y_i\in\RR_+$. For $\alpha=(\alpha_1,\ldots,\alpha_{n-1})\in\CC^{n-1}$, the function $I_\alpha$ on $\frakh^n$ given by
\begin{equation}
 I_\alpha(z) = \prod_{i=1}^{n-1}\prod_{j=1}^{n-1} y_i^{b_{ij}\alpha_j},\label{eq:DefinitionIdelta}
\end{equation}
where
$$ b_{ij} = \begin{cases}ij&\mbox{for $i+j\leq n$,}\\(n-i)(n-j)&\mbox{otherwise,}\end{cases} $$
is a joint eigenfunction of all differential operators in $Z(\frakg)$. A Maass form $f$ is said to be of type $\alpha$ if its eigenvalues $\lambda_D$, $D\in Z(\frakg)$, as defined in \eqref{label:DefMaassForm2}, coincide with the eigenvalues of $I_\alpha(z)$, i.e.
$$ DI_\alpha = \lambda_D I_\alpha,\quad Df=\lambda_Df, \qquad \forall\,D\in Z(\frakg). $$
Note that if $f$ is a Maass form of type $\alpha$, then its complex conjugate $\overline{f}$ is a Maass form of type $\overline{\alpha}$.

\subsection{Fourier--Whittaker expansion}\label{sec:FWexpansion}

We recall the expansion of Maass forms for $\GL(n,\RR)$ into Whittaker functions (see e.g. \cite[Chapters 5 and 9]{Gol06}).

For ${\bf m}=(m_1,\ldots,m_{n-1})\in\ZZ^{n-1}$ let $\psi_{\bf m}$ denote the unitary character of $N_G$ given by
$$ \psi_{\bf m}(x) = e^{2\pi\sqrt{-1}(m_1x_{n-1,n}+\cdots+m_{n-1}x_{1,2})}, \qquad x\in N_G. $$
For ${\bf m}=(1,\ldots,1)$ we will write $\psi=\psi_{1,\ldots,1}$ for short. Then Jacquet's Whittaker function of parameter $\alpha\in\CC^{n-1}$ is defined for $\Re\alpha_i>\frac{1}{n}$ ($i=1,\ldots,n-1$) by the convergent integral
\begin{equation}
 W_J(z;\alpha,\psi_{\bf m}) = \int_{N_G}I_\alpha(wuz)\overline{\psi_{\bf m}(u)}\,du, \qquad z\in\frakh^n,\label{eq:DefinitionWhittaker}
\end{equation}
where $du=\prod_{1\leq i<j\leq n}du_{i,j}$ is a suitably normalized Haar measure on $N_G$ and $w$ the longest Weyl group element with entries $w_{i,j}=\delta_{i,n-j+1}$, extended holomorphically to $\alpha\in\CC^{n-1}$.

The Fourier--Whittaker expansion of a Maass form $f$ of type $\alpha$ is of the form
\begin{multline}
 f(z) = \sum_{\small\substack{\gamma\in\\N_G\cap\SL(n-1,\ZZ)\backslash\SL(n-1,\ZZ)}}\sum_{\small\substack{m_1,\ldots,m_{n-2}\geq 1\\m_{n-1}\neq 0}}\frac{A_f(m_1,\ldots,m_{n-1})}{\prod_{k=1}^{n-1}|m_k|^{\frac{k(n-k)}{2}}}\\
 \times W_J\left(\begin{pmatrix}m_1\cdots m_{n-2}|m_{n-1}|&&&&\\&\ddots&&&\\&&m_1m_2&&\\&&&m_1&\\&&&&1\end{pmatrix}\begin{pmatrix}\gamma&\\&1\end{pmatrix}z;\alpha,\psi_{1,\ldots,1,\frac{m_{n-1}}{|m_{n-1}|}}\right).\label{eq:FourierExpansion}
\end{multline}
with Fourier coefficients $A_f(m_1,\ldots,m_{n-1})\in\CC$.

Since the full ring of Hecke operators commutes with the ring of invariant differential operators on $\Gamma_G\cdot Z(G)\backslash G$, we can further decompose Maass forms into eigenfunctions of all Hecke operators (see \cite[Chapter 9]{Gol06} for details). Maass forms which additionally are Hecke eigenfunctions are called Hecke--Maass forms and they satisfy $f=0$ if and only if $A_f(1,\ldots,1)=0$. Assuming $f\neq0$, we can therefore normalize
$$ \tilde{A}_f(m_1,\ldots,m_{n-1})=\frac{A_f(m_1,\ldots,m_{n-1})}{A_f(1,\ldots,1)}, $$
then $\tilde{A}_f(m,1,\ldots,1)$ are precisely the eigenvalues of the Hecke operators and all other Fourier coefficients $\tilde{A}_f(m_1,\ldots,m_{n-1})$ are uniquely determined by the Hecke eigenvalues $\tilde{A}_f(m,1,\ldots,1)$. From this viewpoint, it might be more natural to study Hecke eigenvalues, but for our purpose the Fourier expansion \eqref{eq:FourierExpansion} is more central, because it allows to unfold the period integral to the $L$-function. We therefore do not use Hecke operators in what follows.

\subsection{Rankin--Selberg $L$-functions}\label{RS}

Fix $1\leq n'\leq n$ and let $f$ be a Hecke--Maass form on $\GL(n,\RR)$ and $g$ be a Hecke--Maass form on $\GL(n',\RR)$ with normalized Fourier coefficients $\tilde{A}_f(m_1,\ldots,m_{n-1})$ and $\tilde{A}_g(m_1,\ldots,m_{n'-1})$, respectively. The Rankin--Selberg $L$-function of $f\times\overline{g}$ is for $n'=n$ defined by 
$$ L(s,f\times\overline{g}) = \zeta(ns)\sum_{m_1,\ldots,m_{n-1}=1}^\infty \frac{\tilde{A}_f(m_1,\ldots,m_{n-1})\overline{\tilde{A}_g(m_1,\ldots,m_{n-1})}}{(m_1^{n-1}m_2^{n-2}\cdots m_{n-1})^s} $$
and for $1\leq n'\leq n-1$ by
$$ L(s,f\times\overline{g}) = \sum_{m_1,\ldots,m_{n'}=1}^\infty \frac{\tilde{A}_f(m_1,\ldots,m_{n'},1,\ldots,1)\overline{\tilde{A}_g(m_2,\ldots,m_{n'})}}{(m_1^{n'}m_2^{n'-1}\cdots m_{n'})^s}. $$
These sums converge for sufficiently large $\Re(s)$ and $L(s,f\times\overline{g})$ has a holomorphic continuation to $s\in\CC$ except in the case $n'=n$ where there might occur a simple pole at $s=1$ (see \cite[Theorem 12.1.4]{Gol06} for details).

In the special case $n'=1$ the Rankin--Selberg $L$-function reduces to the standard Godement--Jacquet $L$-function of $f$, which is defined by
$$ L(s,f) = \sum_{m=1}^\infty\frac{\tilde{A}_f(1,\ldots,1,m)}{m^s}; $$
more precisely, $L(s,f\times 1)=\overline{L(\overline{s},f)}$ since $A(1,\ldots,1,m)=\overline{A(m,1,\ldots,1)}$.

\subsection{The period integral}\label{period}

The relation between the Rankin--Selberg $L$-function $L(s,f\times\overline{g})$ and a period integral of $f$ and $\overline{g}$ differs in the cases $n'=n$, $n'=n-1$ and $1\leq n'\leq n-2$, so we treat these cases separately.

\subsubsection{The case $n'=n$}

For $s\in\CC$ and $z=xy\in\mathfrak{h}^n$ as in \eqref{eq:DefinitionXY} let
$$ I_s(z) = \det(y)^s = y_1^{(n-1)s}y_2^{(n-2)s}\cdots y_{n-1}^s. $$
Following \cite[Chapter 10.4]{Gol06}, we define the degenerate Eisenstein series attached to the standard maximal parabolic subgroup $P_{G,\max}$ of $G$ corresponding to the partition $n=(n-1)+1$ for sufficiently large $\Re s$ by the absolutely convergent sum
$$ E_s(z) = \sum_{\gamma\in P_{G,\max}\cap\SL(n,\ZZ)\backslash\SL(n,\ZZ)} I_s(\gamma z). $$
The degenerate Eisenstein series $E_s(z)$ has a meromorphic continuation to $s\in\CC$. The Rankin--Selberg convolution of two Maass forms $f$ and $g$ on $\GL(n,\RR)$ is for $s\in\CC$ with sufficiently large $\Re s$ defined by the period integral
$$ \Lambda(s,f\times\overline{g}) = \int_{\Gamma_G\backslash\mathfrak{h}^n} f(z)\overline{g(z)}E_s(z)\,d^*z, $$
where $d^*z$ denotes a (suitably normalized) measure on $\Gamma_G\backslash\mathfrak{h}^n$, which is locally given by an $\SL(n,\RR)$-invariant measure on $\mathfrak{h}^n$ (see \cite[Chapter 1.5]{Gol06} for details).

As in \cite[Chapter 12.1]{Gol06} one shows that if $f$ is of type $\alpha\in\CC^{n-1}$ and $g$ is of type $\beta\in\CC^{n-1}$ the Rankin--Selberg convolution can be expressed in terms of the Rankin--Selberg $L$-function as follows:
$$ \zeta(ns)\cdot\Lambda(s,f\times\overline{g}) = A_f(1,\ldots,1)\overline{A_g(1,\ldots,1)}\cdot G_{\alpha,\overline{\beta}}(s)\cdot L(s,f\times\overline{g}), $$
where
\begin{align*}
G_{\alpha,\overline{\beta}}(s) &= \int_{{\RR_+^{n-1}}}W_J\left(y;\alpha,\psi\right)W_J(y;\overline{\beta},\overline{\psi})\cdot|\det(y)|^s\,d^*y,
\end{align*}
with $y=\diag(y_1\cdots y_{n-1},\ldots,y_1y_2,y_1,1)$ and $d^*y=\prod_{j=1}^{n-1}y_j^{-1-j(n-j)}dy_j$.

\subsubsection{The case $n'=n-1$}

Let $f$ be a cuspidal Maass form on $\GL(n,\RR)$. Then its restriction $f|_{\GL(n',\RR)}$ to $\GL(n',\RR)\subseteq\GL(n,\RR)$, embedded as a block in the upper left corner of $\GL(n,\RR)$, decays rapidly on $\GL(n',\RR)\cap A_G$ and therefore the following integral converges for all Maass forms $g$ on $\GL(n',\RR)$ and all $s\in\CC$:
$$ \Lambda(s,f\times\overline{g}) = \int_{\SL(n',\ZZ)\backslash\GL(n',\RR)} f(z)\overline{g(z)}|\det(z)|^{s-\frac{1}{2}}\,d^*z, $$
where $d^*z$ denotes a (suitably normalized) $\GL(n',\RR)$-invariant measure on $\SL(n',\ZZ)\backslash\GL(n',\RR)$. In \cite[Chapter 12.2]{Gol06} it is shown that for $f$ of type $\alpha\in\CC^{n-1}$ and $g$ of type $\beta\in\CC^{n-2}$ we have
$$ \Lambda(s,f\times\overline{g}) = A_f(1,\ldots,1)\overline{A_g(1,\ldots,1)}\cdot G_{\alpha,\overline{\beta}}(s)\cdot L(s,f\times\overline{g}), $$
where
\begin{align*}
	G_{\alpha,\overline{\beta}}(s) &= \int_{{\RR_+^{n-1}}}W_J\left(\begin{pmatrix}y&\\&1\end{pmatrix};\alpha,\psi\right)W_J(y;\overline{\beta},\overline{\psi})\cdot|\det(y)|^{s-\frac{1}{2}}\,d^*y,
\end{align*}
with $y=\diag(y_1\cdots y_{n-1},\ldots,y_1y_2,y_1)$ and $d^*y=\prod_{j=1}^{n-1}y_j^{-1-(j-1)(n-j)}dy_j$.

\subsubsection{The case $1\leq n'\leq n-2$}\label{sec:PeriodIntegralCase3}

Let $Y_{n,n'}\subseteq G$ denote the unipotent radical of the standard parabolic subgroup of $G$ associated to the partition $n=(n'+1)+1+\cdots+1$, i.e. $Y_{n,n'}=\{u\in N_G:u_{ij}=0\mbox{ whenever }1\leq i<j\leq n'+1\}$.

Following \cite[Chapter 12.3]{Gol06} we consider for a Maass form $f$ on $\GL(n,\RR)$ the projection
$$ \PP^n_{n'} f(z) = |\det(z)|^{-\frac{n-n'-1}{2}}\int_{Y_{n,n'}\cap\Gamma_G\backslash Y_{n,n'}}f\left(u\begin{pmatrix}z&\\&\1_{n-n'}\end{pmatrix}\right)\overline{\psi(u)}\,du \qquad (z\in\GL(n',\RR)). $$

The Rankin--Selberg convolution of a cusp form $f$ for $\SL(n,\ZZ)$ with a Maass form $g$ for $\SL(n',\ZZ)$ is for $s\in\CC$ defined by the period integral
$$ \Lambda(s,f\times\overline{g}) = \int_{\SL(n',\ZZ)\backslash\GL(n',\RR)} \PP^n_{n'} f(z)\cdot\overline{g(z)}\cdot|\det(z)|^{s-\frac{1}{2}}\,d^*z, $$
where $d^*z$ is a (suitably normalized) $\GL(n',\RR)$-invariant measure on $\SL(n',\ZZ)\backslash\GL(n',\RR)$.

As in \cite[Chapter 12.3]{Gol06} one shows that if $f$ is of type $\alpha\in\CC^{n-1}$ and $g$ is of type $\beta\in\CC^{n'-1}$ the Rankin--Selberg convolution can be expressed in terms of the Rankin--Selberg $L$-function as follows:
$$ \Lambda(s,f\times\overline{g}) = A_f(1,\ldots,1)\overline{A_g(1,\ldots,1)}\cdot G_{\alpha,\overline{\beta}}(s)\cdot L(s,f\times\overline{g}), $$
where
\begin{align*}
	G_{\alpha,\overline{\beta}}(s) &= \int_{{\RR_+^{n'}}}W_J\left(\begin{pmatrix}y&\\&\1_{n-n'}\end{pmatrix};\alpha,\psi\right)W_J(y;\overline{\beta},\overline{\psi})\cdot|\det(y)|^{s-\frac{n-n'}{2}}\,d^*y,
\end{align*}
with $y=\diag(y_1\cdots y_{n'},\ldots,y_1y_2,y_1)$ and $d^*y=\prod_{j=1}^{n'}y_j^{-1-(j-1)(n'-j+1)}dy_j$.

\subsection{Results about $G_{\alpha,\overline{\beta}}(s)$}

In general, no explicit formula for $G_{\alpha,\overline{\beta}}(s)$ is known. However, for the special cases $n'=n$ and $n'=n-1$ the integral was evaluated explicitly by Stade as a product of gamma factors (see \cite{Sta01,Sta02}). Moreover, for $n'=n-2$ the integral can be simplified to a one-dimensional Barnes integral (see Ishii--Stade~\cite{IS13}). We also refer to Jacquet~\cite{Jac72} for the case $(n,n')=(2,2)$, Stade~\cite{Sta93} for the case $(n,n')=(3,3)$, Bump~\cite{Bum88} for the case $(n,n')=(3,2)$ and Hoffstein--Murty~\cite{HM89} for the case $(n,n')=(3,1)$.

All results are most easily expressed in terms of the Langlands parameters $\lambda\in\CC^n$ and $\nu\in\CC^{n'}$ instead of $\alpha\in\CC^{n-1}$ and $\overline{\beta}\in\CC^{n'-1}$. These are uniquely determined by $\lambda_1+\cdots+\lambda_n=\nu_1+\cdots+\nu_{n'}=0$ and the relations
$$ \alpha_j = \frac{1}{n}(\lambda_j-\lambda_{j+1}+1), \qquad \overline{\beta_k} = \frac{1}{n'}(\nu_k-\nu_{k+1}+1). $$
Abusing notation, we write $G_{\lambda,\nu}(s)$ instead of $G_{\alpha,\overline{\beta}}(s)$.

We further use the standard notation
$$ \Gamma_\RR(s) = \frac{\Gamma(\frac{s}{2})}{\pi^{\frac{s}{2}}}. $$

\begin{theorem}[see \cite{IS13,Sta01,Sta02}]
\begin{enumerate}
\item For $n'=n$ we have
$$ G_{\lambda,\nu}(s) = \frac{\prod_{j,k=1}^n\Gamma_\RR(s+\lambda_j+\nu_k)}{2^{n-1}\Gamma_\RR(ns)\prod_{1\leq j<k\leq n}\Gamma_\RR(\lambda_j-\lambda_k+1)\Gamma_\RR(\nu_j-\nu_k+1)}. $$
\item For $n'=n-1$ we have
$$ G_{\lambda,\nu}(s) = \frac{\prod_{j=1}^n\prod_{k=1}^{n-1}\Gamma_\RR(s+\lambda_j+\nu_k)}{2^{n-1}\prod_{1\leq j<k\leq n}\Gamma_\RR(\lambda_j-\lambda_k+1)\prod_{1\leq j<k\leq n-1}\Gamma_\RR(\nu_j-\nu_k+1)}. $$
\item For $n'=n-2$ we have
\begin{multline*}
 \hspace{1.5cm}G_{\lambda,\nu}(s) = \frac{\prod_{j=1}^n\prod_{k=1}^{n-2}\Gamma_\RR(s+\lambda_j+\nu_k)}{2^{n-1}\prod_{1\leq j<k\leq n}\Gamma_\RR(\lambda_j-\lambda_k+1)\prod_{1\leq j<k\leq n-2}\Gamma_\RR(\nu_j-\nu_k+1)}\\
 \times \frac{1}{2\pi\sqrt{-1}}\int_\gamma \frac{\prod_{j=1}^n\Gamma_\RR(z-\lambda_j)}{\prod_{k=1}^{n-2}\Gamma_\RR(s+z+\nu_k)}\,dz,
\end{multline*}
where $\gamma$ is a contour from $-i\infty$ to $+i\infty$ such that all poles of the integrand are on its left.
\end{enumerate}
\end{theorem}

We remark that the correction factor $2^{n-1}$ appears in the denominator since integration is over $\RR_+$ rather than $\RR^\times$. The correction factor $\Gamma_\RR(ns)$ appears in the denominator when $n'=n$, because the integral is over $A_G$ modulo the center instead of $A_G$. The correction factors $\Gamma_\RR(\lambda_j-\lambda_k+1)$ and $\Gamma_\RR(\nu_j-\nu_k+1)$ appear in the denominator, because the Whittaker functions $W_J$ have been $L^2$-normalized (in contrast to Stade's $W_{(n,\alpha)}$, see below). In this sense, the relevant terms for the completion of the $L$-function are the factors $\Gamma_\RR(s+\lambda_j+\nu_k)$.

\begin{proof}
First note that $W_J(y;\alpha,\psi)=W_J(y;\alpha,\overline{\psi})$ for all $y\in A_G$ (see \cite[Lemma 6.5.6]{Gol06} for the case $n=3$, the same argument works for general $n$). To translate between Goldfeld's notation (which we use, see \cite{Gol06}) and Stade's notation (see \cite{IS13,Sta90,Sta01,Sta02}), we note that the function $I_\alpha(z)$ defined in \eqref{eq:DefinitionIdelta} is related to Stade's $H_\alpha(z)$ defined in \cite[equation (1.1)]{Sta90} by
$$ I_\alpha(z) = H_{\alpha^*}(z) \qquad (z\in\mathfrak{h}^n), $$
where $\alpha^*=(\alpha_{n-1},\ldots,\alpha_1)$. This implies that the Whittaker function $W_J(y;\alpha,\psi)$ defined in \eqref{eq:DefinitionWhittaker} is related to Stade's $W_{(n,\alpha)}(y_1,\ldots,y_{n-1})$ defined in \cite[equation (2.2)]{Sta90} by
$$ W_J(y;\alpha,\psi) = \prod_{1\leq i<j\leq n-1}\frac{1}{\Gamma_\RR(\lambda_i-\lambda_j+1)}W_{(n,\alpha^*)}(y_{n-1},\ldots,y_1). $$
In \cite{IS13,Sta01,Sta02}, Langlands parameters are used instead of the parameter $\alpha$, which yields the relation $W_{(n,\alpha^*)}=W_{n,\lambda}$ between the Whittaker function $W_{(n,\alpha^*)}$ from \cite{Sta90} and the Whittaker function $W_{n,\lambda}$ from \cite{Sta01} (see \cite[top display on p. 132]{Sta01}). With these relations, (a) is \cite[Theorem 1.1]{Sta02}, (b) is \cite[Theorem 3.4]{Sta01} and (c) is \cite[Theorem 3.2]{IS13}.
\end{proof}

\section{Automorphic Rankin--Selberg periods}\label{sec:AutRSperiods}

In this section we explain how the Rankin--Selberg period integrals define invariant forms on tensor products of automorphic representations.

\subsection{Automorphic representations}

Under the unitary action of $G$ on $L^2(\Gamma_G\cdot Z(G)\backslash G)$, any Maass form $f$ on $G$ generates an irreducible subrepresentation of $L^2(\Gamma_G\cdot Z(G)\backslash G)$ that is spherical, the function $f$ being the unique (up to scalar multiples) $K_G$-spherical vector. We write $V_f\subseteq L^2(\Gamma_G\cdot Z(G)\backslash G)$ for the subspace of smooth vectors and $\pi_f$ for the corresponding action of $G$ on $V_f$. Then $V_f\subseteq C^\infty(\Gamma_G\cdot Z(G)\backslash G)$. If further $f$ is a cusp form, it was shown in \cite{MS12} that all functions in $V_f$ decay rapidly along $A_G$.

Similarly, we denote by $(\tau_{\overline{g}},W_{\overline{g}})$ the smooth vectors of the irreducible unitary subrepresentation of $L^2(\SL(n',\ZZ)\cdot\RR^\times\backslash\GL(n',\RR))$ generated by the complex conjugate $\overline{g}$ of a Maass form $g$ on $\GL(n',\RR)$ ($1\leq n'\leq n$).

Moreover, for $s\in\CC$ the degenerate Eisenstein series $E_s\in C^\infty(\Gamma_G\cdot Z(G)\backslash G)$ generates a spherical subrepresentation $U_s\subseteq C^\infty(\Gamma_G\cdot Z(G)\backslash G)$ that is isomorphic to a subrepresentation of a degenerate principal series of $G$ (see Section~\ref{sec:DegeneratePrincipalSeries} for a more detailed description). Here $E_s$ is viewed as a $K_G$-invariant and $Z(G)$-invariant function on $G$. We write $\sigma_s$ for the corresponding $G$-action on $U_s$. Note that for $s\in\CC$ outside a certain discrete set, the degenerate principal series is irreducible, so that $\sigma_s$ is isomorphic to the full degenerate principal series.

\subsection{Automorphic Rankin--Selberg periods}

Let $f$ be a Maass cusp form for $\SL(n,\ZZ)$ and $g$ a Maass form for $\SL(n',\ZZ)$. The period integral $\Lambda(s,f\times\overline{g})$ can be extended to an invariant linear form on the tensor products of the corresponding automorphic representations. For this we again distinguish between the cases $n'=n$, $n'=n-1$ and $1\leq n'\leq n-2$.

\subsubsection{The case $n'=n$}

We have
$$ \Lambda(s,f\times\overline{g}) = \int_{\Gamma_G\cdot Z(G)\backslash G} f(h)\overline{g(h)}E_s(h)\,dh, $$
where $E_s$ is viewed as a $K_G$-invariant function on $\Gamma_G\cdot Z(G)\backslash G$. This integral makes sense if we replace $f$ and $\overline{g}$ by arbitrary functions in $V_f$ and $W_{\overline{g}}$ and $E_s$ by an arbitrary function in $U_s$ and defines a $G$-invariant linear form
$$ \ell^{\rm aut}_{f,\overline{g},s}:V_f\widehat{\otimes} W_{\overline{g}}\widehat{\otimes} U_s\to\CC, \quad v\otimes w\otimes u\mapsto \int_{\Gamma_G\cdot Z(G)\backslash G} v(h)w(h)u(h)\,dh, $$
i.e. $\ell^{\rm aut}_{f,\overline{g},s}\in\Hom_G(\pi_f\widehat{\otimes}\tau_{\overline{g}}\widehat{\otimes}\sigma_s,\CC)$. The period $\Lambda(s,f\times\overline{g})$ can be recovered from $\ell^{\rm aut}_{f,\overline{g},s}$ as the special value at the tensor product of the spherical vectors: $\Lambda(s,f\times\overline{g})=\ell^{\rm aut}_{f,\overline{g},s}(f\otimes\overline{g}\otimes E_s)$. It follows from \cite[Theorem B]{CS15} that the space $\Hom_G(\pi_f\widehat{\otimes}\tau_{\overline{g}}\widehat{\otimes}\sigma_s,\CC)$ is at most one-dimensional, so that $\ell^{\rm aut}_{f,\overline{g},s}$ is proportional to any other non-zero period in $\Hom_G(\pi_f\widehat{\otimes}\tau_{\overline{g}}\widehat{\otimes}\sigma_s,\CC)$. To keep the notation uniform we put $H=G$ in this case.

\subsubsection{The case $n'=n-1$}

Let $H=\GL(n-1,\RR)$ and $\Gamma_H=\SL(n-1,\ZZ)$. Similar to the case $n'=n$ we can write $\Lambda(s,f\times\overline{g})=\ell^{\rm aut}_{f,\overline{g},s}(f\otimes\overline{g})$ for the $H$-invariant linear form
$$ \ell^{\rm aut}_{f,\overline{g},s}:V_f\widehat{\otimes} W_{\overline{g}}\to\CC, \quad v\otimes w\mapsto\int_{\Gamma_H\backslash H}v(h)w(h)|\det(h)|^{s-\frac{1}{2}}\,dh, $$
i.e. $\ell^{\rm aut}_{f,\overline{g},s}\in\Hom_H(\pi_f|_H\widehat{\otimes}\tau_{\overline{g}}\widehat{\otimes}\chi_s,\CC)$ with $\chi_s(h)=|\det(h)|^{s-\frac{1}{2}}$. By \cite[Theorem B]{SZ12} this space is at most one-dimensional, so $\ell^{\rm aut}_{f,\overline{g},s}$ is proportional to any other non-zero period in $\Hom_H(\pi_f|_H\widehat{\otimes}\tau_{\overline{g}}\widehat{\otimes}\chi_s,\CC)$.

\subsubsection{The case $1\leq n'\leq n-2$}

Let
$$ H = \left\{\begin{pmatrix}h_1&u\\&h_2\end{pmatrix}:h_1\in\GL(n',\RR),u\in M(n'\times(n-n'),\RR)^0,h_2\in N_{n-n'}\right\}, $$
where $M(n'\times(n-n'),\RR)$ denotes the space of real $n'\times(n-n')$-matrices, $M(n'\times(n-n'),\RR)^0$ the subspace of those matrices with first column equal to zero, and $N_{n-n'}$ the group of unipotent upper triangular matrices of size $n-n'$. We note that $H\simeq\GL(n',\RR)\ltimes Y_{n,n'}$ with $Y_{n,n'}$ as in Section~\ref{sec:PeriodIntegralCase3}. Let further $\Gamma_H=H\cap\Gamma_G$.

The Rankin--Selberg convolution $\Lambda(s,f\times\overline{g})$ can be written as
$$ \Lambda(s,f\times\overline{g}) = \int_{\Gamma_H\backslash H}f(h)\overline{g(h)}\cdot\chi_s(h)\,dh, $$
where
\begin{equation}
 \chi_s\begin{pmatrix}h_1&u\\&h_2\end{pmatrix} = |\det(h_1)|^{s-\frac{n-n'}{2}}\overline{\psi(h_2)}.\label{eq:DefinitionChiS}
\end{equation}
Here we extend $g$ trivially to $H$ by putting $g(h):=g(h_1)$. Moreover, $dh$ denotes the right-invariant measure on $\Gamma_H\backslash H$ given by
$$ \int_{\Gamma_H\backslash H} \varphi(h)\,dh = \int_{\SL(n',\ZZ)\backslash\GL(n',\RR)}\int_{Y_{n,n'}\cap\Gamma_G\backslash Y_{n,n'}}\varphi(vh_1)\,dv\,dh_1. $$
(Note that $H$ is not unimodular.)

The integral defining $\Lambda(s,f\times\overline{g})$ makes sense even if we replace $f$ by an arbitrary function in $V_f$ and $\overline{g}$ by an arbitrary function in $W_{\overline{g}}$, which leads to the map
$$ \ell^{\rm aut}_{f,\overline{g},s}:V_f\widehat{\otimes} W_{\overline{g}}\to\CC, \quad v\otimes w\mapsto\int_{\Gamma_H\backslash H}v(h)w(h)\chi_s(h)\,dh, $$
which is defined due to the rapid decay of $u\in V_f$. It respects the action of the subgroup $H$ in the sense that $\ell^{\rm aut}_{f,\overline{g},s}$ intertwines the representation $\pi_f|_H\widehat{\otimes}\tau_{\overline{g}}\widehat{\otimes}\chi_s$ on $V_f\widehat{\otimes} W_{\overline{g}}$ and the trivial representation on $\CC$, i.e. $\ell^{\rm aut}_{f,\overline{g},s}\in\Hom_H(\pi_f|_H\widehat{\otimes}\tau_{\overline{g}}\widehat{\otimes}\chi_s,\CC)$. In \cite[Theorem A]{CS15} it was shown that $\dim\Hom_H(\pi_f|_H\widehat{\otimes}\tau_{\overline{g}}\widehat{\otimes}\chi_s,\CC)\leq1$ for all $s\in\CC$, so that the automorphic period $\ell_{f,\overline{g},s}^{\rm aut}$ is proportional to any other period in $\Hom_H(\pi_f|_H\widehat{\otimes}\tau_{\overline{g}}\widehat{\otimes}\chi_s,\CC)$.

\section{Model Rankin--Selberg periods on spherical principal series}\label{sec:ModRSperiods}

In this section we explicitly construct Rankin--Selberg periods on tensor products of spherical principal series representations.

\subsection{Parabolic subgroups}

Let $P_G$ denote the standard minimal parabolic subgroup of $G$ consisting of all upper triangular matrices. The parabolic subgroup $P_G$ has a Langlands decomposition $P_G=M_GA_GN_G$ with $M_G$ the subgroup of diagonal matrices with entries in $\{\pm1\}$ and $A_G$ and $N_G$ as in Section~\ref{sec:MaassForms}.

In the case $n'=n-1$ the group $P_H=P_G\cap H$ is the standard minimal parabolic subgroup of $H$ and $P_H=M_HA_HN_H$ with $M_H=M_G\cap H$, $A_H=A_G\cap H$, $N_H=N_G\cap H$ is a Langlands decomposition.

For $n'=n-2$ we also put $P_H=P_G\cap H$, which has a similar decomposition $P_H=M_HA_HN_H$, where $M_H\simeq M_G\cap\GL(n',\RR)$, $A_H\simeq A_G\cap\GL(n',\RR)$ and $N_H\simeq(N_{n'}\times N_{n-n'})\ltimes M(n'\times(n-n'),\RR)^0$.

\subsection{Principal series representations}

For $\lambda\in\CC^n$ and $a=\diag(a_1,\ldots,a_n)\in G$ let
$$ a^\lambda =a_1^{\lambda_1}\cdots a_n^{\lambda_n}. $$
This defines a character $e^\lambda:A_G\to\CC^\times,\,a\mapsto a^\lambda$. Let
$$ \pi_\lambda = \Ind_{P_G}^G(\1\otimes e^\lambda\otimes\1), $$
where we use smooth normalized parabolic induction, i.e. $\pi_\lambda$ can be realized as the right-regular representation of $G$ on
$$ I_\lambda = \{u\in C^\infty(G):u(mang)=a^{\lambda+\rho_G}u(g)\,\forall\,g\in G,man\in M_GA_GN_G\} $$
with $\rho_G=(\frac{n-1}{2},\frac{n-3}{2},\ldots,\frac{1-n}{2})$ corresponding to the half sum of all roots of $\fraka_G$ in $\frakn_G$. In particular, $\pi_\lambda$ is unitary for $\lambda\in(i\RR)^n$, the invariant inner product being
$$ \langle u,v\rangle = \int_{K_G} u(k)\overline{v(k)}\,dk. $$

The representation $\pi_\lambda$ is $K_G$-spherical and we normalize the spherical vector $\phi_\lambda\in I_\lambda$ such that $\phi_\lambda(e)=1$. To give an explicit formula for $\phi_\lambda$ we define for $1\leq k\leq n$ and $1\leq i_1,\ldots,i_k\leq n$ a polynomial $p_{i_1,\ldots,i_k}$ on $M(n\times n,\RR)$ by
\begin{equation}
 p_{i_1,\ldots,i_k}(x) := \det((x_{ij})_{i=n,\ldots,n-k+1}^{j=i_1,\ldots,i_k}).\label{eq:DefPolynomials}
\end{equation}
For $1\leq k\leq n$ we define polynomials $\Phi_k$ on $M(n\times n,\RR)$ by
\begin{equation}
 \Phi_k(x) = p_{1,\ldots,k}(x)\label{eq:DefinitionPhi}
\end{equation}
and note that they satisfy
\begin{equation}
	\Phi_k(gxh) = g_{n,n}\cdots g_{n-k+1,n-k+1}h_{1,1}\cdots h_{k,k}\Phi_k(x) \qquad \forall\,g,h\in P_G.\label{eq:EquivPhik}
\end{equation}
We remark that $\Phi_k$ also can be expressed as the determinant of a matrix product:
$$ \Phi_k(x) = \det\left(\begin{pmatrix}\1_{k\times k}&\0_{k\times(n-k)}\end{pmatrix}wx\begin{pmatrix}\1_{k\times k}\\\0_{(n-k)\times k}\end{pmatrix}\right), $$
where $w$ represents the longest Weyl group element (see Section \ref{sec:FWexpansion}).

\begin{lemma}
	$\phi_\lambda$ can be expressed as
	\begin{equation}
		\phi_\lambda(g) = \prod_{k=1}^n\big|\Phi_k(gg^\top w)\big|^{\frac{\lambda_{n-k+1}-\lambda_{n-k}-1}{2}} \qquad (g\in G),\label{eq:ExplicitFormulaSphericalVector}
	\end{equation}
	where $w$ is a representative of the longest Weyl group element (see Section \ref{sec:FWexpansion}).
\end{lemma}

\begin{proof}
	$\phi_\lambda$ is the unique smooth function on $G$ such that $\phi_\lambda(nak)=a^{\lambda+\rho_G}$, so it suffices to show that the right hand side of \eqref{eq:ExplicitFormulaSphericalVector} has the same properties. Since $gg^\top$ is positive definite, its principal minors are strictly positive. This implies that $g\mapsto\Phi_k(gg^\top w)$ is a smooth nowhere vanishing function on $G$. Hence, the right hand side of \eqref{eq:ExplicitFormulaSphericalVector} is smooth. We further have
	\begin{equation}
		\Phi_k(nak) = \Phi_k(nakk^\top a^\top n^\top w) = \Phi_k(na^2w(w^{-1}n^\top w)) = a_{n,n}\cdots a_{n-k+1,n-k+1}\Phi_k(w)\label{eq:PhikNAK}
	\end{equation}
	for all $k\in K_G$, $a\in A_G$ and $n\in N_G$ by \eqref{eq:EquivPhik}. Note that $\Phi_k(w)=\pm1$, depending on the choice of the representative $w$ of the longest Weyl group element. (One possible choice is $w_{i,j}=\delta_{i,n-j+1}$, then $\Phi_k(w)=1$ for all $k=1,\ldots,n$.) Applying \eqref{eq:PhikNAK} to every factor in \eqref{eq:ExplicitFormulaSphericalVector} shows the claim.
\end{proof}

Similarly, for $\nu\in\CC^{n'}$ we write $\tau_\nu$ for the corresponding spherical principal series representation of $\GL(n',\RR)$ and realize $\tau_\nu$ on a subspace $J_\nu\subseteq C^\infty(\GL(n',\RR))$. For $n'=n$ and $n'=n-1$ this defines a representation of $H=\GL(n',\RR)$, and in the case $1\leq n'\leq n-2$ we extend $\tau_\nu$ trivially to $H=(\GL(n',\RR)\times N_{n-n'})\ltimes M(n'\times(n-n'),\RR)^0$. For $1\leq n'\leq n-1$ we further recall the character $\chi_s$ of $H$ from \eqref{eq:DefinitionChiS}. Note that $\tau_\nu\otimes\chi_s\simeq\Ind_{P_H}^H(e^\nu\otimes\chi_s)$, realized on
$$ \{u\in C^\infty(H):u(ph)=a^{\nu+\rho_H}\chi_s(p)u(h)\,\forall\,h\in H,p=man\in P_H\} $$
with $\rho_H=(\frac{n'-1}{2},\frac{n'-3}{2},\ldots,\frac{1-n'}{2})$.

\subsection{Degenerate principal series representations}\label{sec:DegeneratePrincipalSeries}

Let $P_{G,\max}\subseteq G$ be the standard maximal parabolic subgroup of $G$ corresponding to the partition $n=(n-1)+1$, i.e.
$$ P_{G,\max} = \left\{\begin{pmatrix}A&b\\0&d\end{pmatrix}:A\in\GL(n-1,\RR),b\in\RR^{n-1},d\in\GL(1,\RR)\right\}. $$
For $r\in\CC$ let
$$ \xi_r\begin{pmatrix}A&b\\0&d\end{pmatrix} := |\det A|^r|d|^{-(n-1)r} $$
and induce this character of $P_{G,\max}$ to $G$ (smooth normalized parabolic induction):
$$ \varsigma_r := \Ind_{P_{G,\max}}^G(\xi_r), $$
realized on the space
$$ L_r = \{f\in C^\infty(G):f(pg)=\xi_{r+\frac{1}{2}}(p)f(g)\,\forall\,g\in G,p\in P_{G,\max}\}. $$
Note that $\varsigma_r$ is unitary for $r\in i\RR$.

Write $f_r$ for the unique $K_G$-spherical vector in $L_r$ with $f_r(e)=1$. We claim that $f_r=I_s$ for $s=r+\frac{1}{2}$. In fact, let $g=xyk\in N_GA_GK_G$, then
$$ f_r(xyk)=\xi_{r+\frac{1}{2}}(xy)=\xi_{r+\frac{1}{2}}(x)\xi_{r+\frac{1}{2}}(y). $$
On $x\in N_G$ the character $\xi_{r+\frac{1}{2}}$ is obviously trivial, and on $y\in A_G$ it is given by
$$ \xi_{r+\frac{1}{2}}(y) = I_s(y). $$

It follows that for $s=r+\frac{1}{2}$ the $G$-intertwining operator
$$ \zeta:L_r \to U_s, \quad f\mapsto \sum_{\gamma\in P_{G,\max}\cap\SL(n,\ZZ)\backslash\SL(n,\ZZ)}f(\gamma g) $$
maps the spherical vector $f_r$ to the degenerate Eisenstein series $E_s$.

An explicit expression for $f_r$ is given by
$$ f_r(g) = |\Phi_1(gg^\top w)|^{-\frac{n}{2}(r+\frac{1}{2})}|\Phi_n(gg^\top w)|^{\frac{1}{2}(r+\frac{1}{2})}. $$

\subsection{Invariant forms on principal series}

We construct invariant forms on tensor products of principal series representations.

\subsubsection{The case $n'=n$}

An invariant form $\ell\in\Hom_G(\pi_\lambda\widehat{\otimes}\tau_\nu\widehat{\otimes}\varsigma_r,\CC)$ is a $G$-invariant continuous linear operator $I_\lambda\widehat{\otimes} I_\nu\widehat{\otimes} L_r\to\CC$ and hence given by a distribution kernel $K\in\calD'(G\times G\times G)$ that satisfies the following equivariance conditions (see e.g. \cite[Sections 3.2 and 3.5]{KS15} or \cite[Section 3.2]{Moe17} for details on this matter):
\begin{enumerate}
\item $K(g_1g,g_2g,g_3g)=K(g_1,g_2,g_3)$ for all $g\in G$,
\item $K(mang_1,g_2,g_3)=a^{-\lambda+\rho_G}K(g_1,g_2,g_3)$ for all $man\in P_G=M_GA_GN_G$,
\item $K(g_1,mang_2,g_3)=a^{-\nu+\rho_G}K(g_1,g_2,g_3)$ for all $man\in P_G=M_GA_GN_G$,
\item $K(g_1,g_2,pg_3)=\xi_{-r+\frac{1}{2}}(p)K(g_1,g_2,g_3)$ for all $p\in P_{G,\max}$.
\end{enumerate}
If $K$ is a locally integrable function, then the corresponding invariant form is given by
\begin{align*}
 \ell(v\otimes w\otimes u) &= \int_{K_G\times K_G\times K_G} K(k_1,k_2,k_3)v(k_1)w(k_2)u(k_3)\,d(k_1,k_2,k_3),
\end{align*}
otherwise the integral has to be understood in the sense of generalized functions. Using the integral formula \cite[formula (5.25)]{Kna01} this integral can be rewritten as
$$ \ell(v\otimes w\otimes u) = \int_{\overline{N}_G\times\overline{N}_G\times\overline{N}_{G,\max}} K(\overline{n}_1,\overline{n}_2,\overline{n}_3)v(\overline{n}_1)w(\overline{n}_2)u(\overline{n}_3)\,d(\overline{n}_1,\overline{n}_2,\overline{n}_3), $$
where $\overline{N}_G$ resp. $\overline{N}_{G,\max}$ denotes the nilradical of the parabolic subgroup of $G$ opposite to $P_G$ resp. $P_{G,\max}$.

Non-zero invariant forms on $\pi_\lambda\widehat{\otimes}\tau_\nu\widehat{\otimes}\varsigma_r$ can only exist if the center of $G$ acts trivially, which implies
$$ \lambda_1+\cdots+\lambda_n+\nu_1+\cdots+\nu_n=0. $$
We will assume this for the construction.

Recall the polynomials $\Phi_k(x)$ on $M(n\times n,\RR)$ from \eqref{eq:DefinitionPhi}. For $x,y\in M(n\times n,\RR)$ and $1\leq k\leq n$ we further define
\begin{equation*}
 \Phi_k(x,y) = \det\begin{pmatrix}x_{n,\ldots,k+1}\\y_{n,\ldots,n-k+2}\\0\,\,\cdots\,\,0\,\,1\end{pmatrix},
\end{equation*}
where for $1\leq i_1,\ldots,i_k\leq n$ we write $x_{i_1,\ldots,i_k}$ for the $k\times n$-matrix whose rows are the rows $i_1,\ldots,i_k$ of $x$, and similar for $y$. An alternative expression for $\Phi_k(x,y)$ is
$$ \Phi_k(x,y) = \det\left(\begin{pmatrix}\begin{pmatrix}\1_{(n-k)\times (n-k)}&\0_{(n-k)\times k}\end{pmatrix}wx\\\begin{pmatrix}\1_{(k-1)\times (k-1)}&\0_{(k-1)\times(n-k+1)}\end{pmatrix}wy\end{pmatrix}\begin{pmatrix}\1_{(n-1)\times(n-1)}\\\0_{1\times(n-1)}\end{pmatrix}\right). $$
It is easy to see that
\begin{align*}
	\Phi_k(gx,hy) &= g_{n,n}\cdots g_{k+1,k+1}h_{n,n}\cdots h_{n-k+2,n-k+2}\Phi_k(x,y) && \forall\,g,h\in P_G,\\
	\Phi_k(xp,yp) &= \det(A)\Phi_k(x,y) && \forall\,p=\begin{pmatrix}A&b\\0&d\end{pmatrix}\in P_{G,\max}.
\end{align*}

Define
\begin{multline*}
	K_{\lambda,\nu,s}(g_1,g_2,g_3) := |\det(g_1g_2)|^\mu|\det(g_3)|^{-2\mu}|\Phi_1(g_1g_2^{-1})|^{s_1}\cdots|\Phi_n(g_1g_2^{-1})|^{s_n}\\
	\times|\Phi_1(g_1g_3^{-1},g_2g_3^{-1})|^{t_1}\cdots|\Phi_n(g_1g_3^{-1},g_2g_3^{-1})|^{t_n}.
\end{multline*}
This function obviously satisfies (1), and a short computation shows that it satisfies (2), (3) and (4) if and only if
\begin{align*}
	\mu &= -\frac{n-1}{2}(s-1), & s_i &= -\lambda_{n-i+1}-\nu_{i+1}-s & t_i &= \lambda_i+\nu_{n-i+1}+s-1.
\end{align*}

Since $\det(x)$, $\Phi_k(x)$ and $\Phi_k(x,y)$ are polynomial functions, the kernel $K_{\lambda,\nu,s}$ is a product of complex powers of analytic functions on $G$ and therefore extends to a meromorphic family of distributions by standard arguments. This shows:

\begin{proposition}\label{prop:ModelFormsEqual}
The prescription
\begin{multline*}
 \ell_{\lambda,\nu,s}^{\mod}:I_\lambda\widehat{\otimes} I_\nu\widehat{\otimes} L_{s-\frac{1}{2}}\to\CC,\\
 v\otimes w\otimes u\mapsto\int_{K_G\times K_G\times K_G}K_{\lambda,\nu,s}(g_1,g_2,g_3)v(g_1)w(g_2)u(g_3)\,d(g_1,g_2,g_3)
\end{multline*}
defines a meromorphic family of invariant forms $\ell_{\lambda,\nu,s}^{\mod}\in\Hom_G(\pi_\lambda\widehat{\otimes}\pi_\nu\widehat{\otimes}\varsigma_{s-\frac{1}{2}},\CC)$.
\end{proposition}

\subsubsection{The case $1\leq n'\leq n-1$}

As in the case $n'=n$ invariant forms $\ell\in\Hom_H(\pi_\lambda|_H\widehat{\otimes}\tau_\nu\widehat{\otimes}\chi_s,\CC)$ correspond to distribution kernels $K\in\calD'(G\times H)$ such that
\begin{enumerate}
\item $K(gk,hk)=K(g,h)$ for all $k\in H$,
\item $K(p_Gg,h)=p_G^{-\lambda+\rho_G}K(g,h)$ for all $p_G\in P_G$,
\item $K(g,p_Hh)=p_H^{-\nu+\rho_H}\chi_s(p_H)^{-1}K(g,h)$ for all $p_H\in P_H$.
\end{enumerate}

Recall the polynomials $\Phi_k(x)$ on $M(n\times n,\RR)$ from \eqref{eq:DefinitionPhi}. We further define
\begin{align*}
 \Psi_k(x) &= p_{1,\ldots,k-1,n'+1}(x) && (1\leq k\leq n'),\\
 \Xi_k(x) &= p_{1,\ldots,k-1,k+1}(x) && (n'+1\leq k\leq n-1),
\end{align*}
where $p_{i_1,\ldots,i_k}(x)$ denote the polynomials from \eqref{eq:DefPolynomials}. Note that
$$ \Psi_k(x) = \Phi_k(xw_{k,n'+1}) \qquad \mbox{and} \qquad \Xi_k(x) = \Phi_k(xw_{k,k+1}), $$
where $w_{i,j}$ denotes the permutation matrix associated to the transposition $(i\,\,j)$.
The following equivariance properties for $g\in P_G$ and $h\in P_H$ are easy to verify:
\begin{align*}
 \Psi_k(gxh) &= g_{n,n}\cdots g_{n-k+1,n-k+1}h_{1,1}\cdots h_{k-1,k-1}\Psi_k(x),\\
 \Xi_k(gxh) &= g_{n,n}\cdots g_{n-k+1,n-k+1}h_{1,1}\cdots h_{n',n'}\big(\Xi_k(x)+h_{k,k+1}\Phi_k(x)\big).
\end{align*}

For $\lambda\in\CC^n$, $\nu\in\CC^{n'}$ and $s\in\CC$ we define the following kernel function:
\begin{multline*}
 K_{\lambda,\nu,s}(g,h) := |\Phi_1(gh^{-1})|^{s_1}\cdots|\Phi_n(gh^{-1})|^{s_n}|\Psi_1(gh^{-1})|^{t_1}\cdots|\Psi_{n'}(gh^{-1})|^{t_{n'}}\\
 \times\exp\Bigg(-2\pi\sqrt{-1}\sum_{k=n'+1}^{n-1}\frac{\Xi_k(gh^{-1})}{\Phi_k(gh^{-1})}\Bigg),
\end{multline*}
where
$$ s_i = \lambda_{n-i}+\nu_i+s-1 \quad (1\leq i\leq n'), \quad s_i = \lambda_{n-i}-\lambda_{n-i+1}-1 \quad (n'+1\leq i\leq n-1), $$
$$ s_n = -\lambda_1+\tfrac{n-1}{2}, \quad t_j = -\lambda_{n-j+1}-\nu_j-s \quad (1\leq j\leq n'). $$
By the above equivariance properties for $\Phi_k$, $\Psi_k$ and $\Xi_k$ it is easy to verify that the function $K_{\lambda,\nu,s}$ satisfies the desired equivariance properties. Further, for $\Re(s_i),\Re(t_j)\geq0$ it is is locally integrable and hence defines a distribution in $\calD'(G\times H)$.

\begin{proposition}\label{prop:ModelFormsUnequal}
	The prescription
	$$ \ell_{\lambda,\nu,s}^{\mod}:I_\lambda\widehat{\otimes} J_\nu\to\CC, \quad v\otimes w\mapsto\int_{K_G\times K_H}K_{\lambda,\nu,s}(g,h)v(g)w(h)\,d(g,h) $$
	defines a meromorphic family of invariant forms $\ell_{\lambda,\nu,s}^{\mod}\in\Hom_H(\pi_\lambda|_H\widehat{\otimes}\tau_\nu\widehat{\otimes}\chi_s,\CC)$.
\end{proposition}

\begin{proof}
Using resolution of singularities the meromorphic continuation can be reduced to that of the distributions
$$ u_\lambda(x) = |x_1|^{\lambda_1}\cdots|x_n|^{\lambda_n}\exp\big(ix_1^{k_1}\cdots x_n^{k_n}\big) \qquad (x\in\RR^n) $$
for fixed $k_1,\ldots,k_n\in\ZZ$, which is discussed in \cite[Section 4.6 and 4.10]{EKP02}.
\end{proof}

\subsection{Relation between automorphic and model periods}

Let $f$ be a cusp form on $\GL(n,\RR)$ and $g$ a Maass form on $\GL(n',\RR)$. Then $\ell_{f,\overline{g},s}^{\aut}\in\Hom_G(\pi_f\widehat{\otimes}\tau_{\overline{g}}\widehat{\otimes}\sigma_s,\CC)$ for $n'=n$ and $\ell_{f,\overline{g},s}^{\aut}\in\Hom_H(\pi_f|_H\widehat{\otimes}\tau_{\overline{g}}\widehat{\otimes}\chi_s,\CC)$ for $1\leq n'\leq n-1$. By the Multiplicity One Theorems \cite[Theorems A and B]{CS15} these space are at most one-dimensional. Now $\pi_f\simeq\pi_\lambda$ and $\tau_{\overline{g}}\simeq\tau_\nu$, where $\lambda\in\CC^n$ and $\nu\in\CC^{n'}$ are the Langlands parameters of $f$ and $\overline{g}$. Let $\theta:I_\lambda\to V_f$ and $\eta:J_\nu\to W_{\overline{g}}$ be equivariant unitary isomorphisms and recall the equivariant isomorphism $\zeta:L_{s-\frac{1}{2}}\to U_s$. Then for $n'=n$
\begin{equation*}
 \ell_{f,\overline{g},s}^{\aut}\circ(\theta\widehat{\otimes}\eta\widehat{\otimes}\zeta)\in\Hom_G(\pi_\lambda\widehat{\otimes}\tau_\nu\widehat{\otimes}\varsigma_{s-\frac{1}{2}},\CC)
\end{equation*}
and for $1\leq n'\leq n-1$
\begin{equation*}
 \ell_{f,\overline{g},s}^{\aut}\circ(\theta\widehat{\otimes}\eta)\in\Hom_H(\pi_\lambda|_H\widehat{\otimes}\tau_\nu\widehat{\otimes}\chi_s,\CC).
\end{equation*}
Using the Multiplicity One Theorems we can therefore relate the automorphic periods to the model periods $\ell_{\lambda,\nu,s}^{\mod}$ constructed in the previous section. There exists a proportionality constant $b_{f,\overline{g},s}\in\CC$ such that for $n'=n$
\begin{equation}
 \Lambda(s,f\times\overline{g}) = \ell_{f,\overline{g},s}^{\aut}(\theta_v\otimes\eta_w\otimes\zeta_u) = b_{f,\overline{g},s}\cdot\ell_{\lambda,\nu,s}^{\mod}(v\otimes w\otimes u) \qquad (v\in I_\lambda,w\in J_\nu,u\in L_{s-\frac{1}{2}})\label{eq:Proportionality1}
\end{equation}
and for $1\leq n'\leq n-1$
\begin{equation}
 \Lambda(s,f\times\overline{g}) = \ell_{f,\overline{g},s}^{\aut}(\theta_v\otimes\eta_w) = b_{f,\overline{g},s}\cdot\ell_{\lambda,\nu,s}^{\mod}(v\otimes w) \qquad (v\in I_\lambda,w\in J_\nu).\label{eq:Proportionality2}
\end{equation}

By the equivariance of $\theta$ and $\eta$, the spherical vectors $\phi_\lambda\in I_\lambda$ and $\psi_\nu\in J_\nu$ are mapped to scalar multiples of the spherical vectors $f\in V_f$ and $\overline{g}\in W_{\overline{g}}$. If we assume that $f$ and $\overline{g}$ are normalized to have $L^2$-norm one, then the respective scalars are of modulus one. Using the fact that $\zeta$ maps the spherical vector $f_{s-\frac{1}{2}}$ to the Eisenstein series $E_s$, it follows that for $n'=n$:
$$ |\Lambda(s,f\times\overline{g})| = |b_{f,\overline{g},s}|\cdot|\ell_{\lambda,\nu,s}^{\mod}(\phi_\lambda\otimes\psi_\nu\otimes f_{s-\frac{1}{2}})| $$
and for $1\leq n'\leq n-1$:
$$ |\Lambda(s,f\times\overline{g})| = |b_{f,\overline{g},s}|\cdot|\ell_{\lambda,\nu,s}^{\mod}(\phi_\lambda\otimes\psi_\nu)|. $$
To estimate $\Lambda(s,f\times\overline{g})$ it therefore suffices to estimate the special values of the model periods at the spherical vectors and the proportionality constants $b_{f,\overline{g},s}$. In this work we focus on the special values of the model periods and hope to come back to the proportionality scalars in a subsequent paper.

\subsection{Special values of invariant forms}

The special values of the model periods are given by integrating the previously constructed distribution kernels against the spherical vectors. In this section we explain how to reduce the number of variables in the integrals and simplify the distribution kernels.

\subsubsection{The case $n'=n$}

We have the following expression for the special value of $\ell^{\mod}_{\lambda,\nu,s}$:

\begin{lemma}\label{lem:ModelIntegral1}
For all $\lambda,\nu\in\CC^n$ and $r\in\CC$ we have
\begin{multline*}
 \ell^{\mod}_{\lambda,\nu,s}(\phi_\lambda\otimes\phi_\nu\otimes f_{s-\frac{1}{2}}) = \int_{\overline{N}_G\times\overline{N}_{G,\max}} \phi_\lambda(\overline{n}_1)f_{s-\frac{1}{2}}(\overline{n}_3)\\
 \times\prod_{k=1}^{n-1}|p_{k,\ldots,n-1}(\overline{n}_1\overline{n}_3^{-1})|^{\lambda_k+\nu_{n-k+1}+s-1}\,d(\overline{n}_1,\overline{n}_3).
\end{multline*}
\end{lemma}

\begin{proof}
Let $r=s-\frac{1}{2}$ for short. Since $\Hom_G(\pi_\lambda\widehat{\otimes}\tau_\nu\widehat{\otimes}\varsigma_r,\CC)\simeq\Hom_G(\pi_\lambda\widehat{\otimes}\varsigma_r,\tau_{-\nu})$ we can write the invariant form $\ell^{\mod}_{\lambda,\nu,s}$ as
$$ \ell^{\mod}_{\lambda,\nu,s}(v\otimes w\otimes u) = \int_{K_G} A_{\lambda,-\nu,r}(v\otimes u)(k)w(k)\,dk = \int_{\overline{N}_G} A_{\lambda,-\nu,r}(v\otimes u)(\overline{n})w(\overline{n})\,d\overline{n}, $$
where $A_{\lambda,-\nu,r}\in\Hom_G(\pi_\lambda\widehat{\otimes}\varsigma_r,\tau_{-\nu})$ is given by
\begin{align*}
 A_{\lambda,-\nu,r}(v\otimes u)(g) &= \int_{K_G\times K_G} K_{\lambda,\nu,s}(k_1,g,k_3)v(k_1)u(k_3)\,d(k_1,k_3)\\
 &= \int_{\overline{N}_G\times\overline{N}_{G,\max}} K_{\lambda,\nu,s}(\overline{n}_1,g,\overline{n}_3)v(\overline{n}_1)u(\overline{n}_3)\,d(\overline{n}_1,\overline{n}_3).
\end{align*}
Now, since $\phi_\lambda\in I_\lambda$ and $f_r\in L_r$ are both $K_G$-invariant, their tensor product $\phi_\lambda\otimes f_r\in I_\lambda\widehat{\otimes} L_r$ is also $K_G$-invariant and is therefore mapped to a $K_G$-invariant vector in $J_{-\nu}$ by the equivariant map $A_{\lambda,-\nu,r}$. The space of $K_G$-invariant vectors in $J_{-\nu}$ is one-dimensional and spanned by $\phi_{-\nu}$, so that
$$ A_{\lambda,-\nu,r}(\phi_\lambda\otimes f_r)(g) = A_{\lambda,-\nu,r}(\phi_\lambda\otimes f_r)(k_0)\cdot \phi_{-\nu}(g) $$
for any $k_0\in K_G$ since $\phi_{-\nu}|_{K_G}=1$. It follows that
$$ \ell^{\mod}_{\lambda,\nu,s}(\phi_\lambda\otimes\phi_\nu\otimes f_r) = A_{\lambda,-\nu,r}(\phi_\lambda\otimes f_r)(k_0)\cdot\int_{K_G}\phi_\nu(k)\phi_{-\nu}(k)\,dk. $$
Since $\phi_\nu|_{K_G}=\phi_{-\nu}|_{K_G}=1$ the latter integral is equal to $1$ and we have
\begin{multline*}
 \ell^{\mod}_{\lambda,\nu,s}(\phi_\lambda\otimes\phi_\nu\otimes f_r) = A_{\lambda,-\nu,r}(\phi_\lambda\otimes f_r)(k_0)\\
 = \int_{\overline{N}_G\times\overline{N}_{G,\max}} \phi_\lambda(\overline{n}_1)f_r(\overline{n}_3)K_{\lambda,\nu,s}(\overline{n}_1,k_0,\overline{n}_3)\,d(\overline{n}_1,\overline{n}_3).
\end{multline*}
To have a simple expression for $K_{\lambda,\nu,s}(\overline{n}_1,k_0,\overline{n}_3)$ we choose $k_0=w$, a representative of the longest Weyl group element. It is easy to see that for $\overline{n}_1\in\overline{N}_G$ we have
$$ \Phi_k(\overline{n}_1w^{-1}) = 1 \qquad \forall\,1\leq k\leq n. $$
Further, a short computation reveals that for $\overline{n}_3=\begin{pmatrix}\1_{n-1}&\\x^\top&1\end{pmatrix}\in\overline{N}_{G,\max}$ we have
\begin{align*}
 \Phi_k(x,w\overline{n}_3^{-1}) &= (-1)^{(k-1)(n-k)}p_{k,\ldots,n-1}(x)
\end{align*}
This shows:
\begin{align*}
 K_{\lambda,\nu,s}(\overline{n}_1,w,\overline{n}_3) = \prod_{k=1}^{n-1}|p_{k,\ldots,n-1}(\overline{n}_1\overline{n}_3^{-1})|^{\lambda_k+\nu_{n-k+1}+s-1}
\end{align*}
and the proof is complete.
\end{proof}

\subsubsection{The case $1\leq n'\leq n-1$}

Making use of the isomorphism $\Hom_H(\pi_\lambda|_H\widehat{\otimes}\tau_\nu\widehat{\otimes}\chi_s,\CC)\simeq\Hom_H(\pi_\lambda|_H\widehat{\otimes}\chi_s,\tau_{-\nu})$ and proceeding as in the previous section shows:

\begin{lemma}\label{lem:ModelIntegral2}
For all $\lambda\in\CC^n$, $\nu\in\CC^{n'}$ and $s\in\CC$ we have
\begin{multline*}
 \ell^{\mod}_{\lambda,\nu,s}(\phi_\lambda\otimes\psi_\nu) = \int_{\overline{N}_G}\phi_\lambda(\overline{n})\prod_{k=1}^{n'}|p_{n,\ldots,n-k+2,n-n'}(\overline{n})|^{-\lambda_{n-k+1}-\nu_k-s}\\
 \times\exp\left(-2\pi\sqrt{-1}\sum_{k=n'+1}^{n-1}p_{n,\ldots,n-k+2,n-k}(\overline{n})\right)\,d\overline{n}.
\end{multline*}
\end{lemma}

\begin{proof}
For most parts of the proof we refer to the previous section. We only remark that for $\overline{n}\in\overline{N}_G$ we have:
\begin{equation*}
 \Phi_k(\overline{n}w^{-1}) = 1, \qquad \Psi_k(\overline{n}w^{-1}) = p_{n,\ldots,n-k+2,n-n'}(\overline{n}), \qquad \Xi_k(\overline{n}w^{-1}) = p_{n,\ldots,n-k+2,n-k}(\overline{n}).\qedhere
\end{equation*}
\end{proof}

\subsection{A conjecture for the model invariant forms}

We conjecture that the special values of our model invariant forms at the spherical vectors behave, in the case of purely imaginary parameters, like the archimedean local $L$-factors $G_{\lambda,\nu}(s)$. More precisely:

\begin{conjecture}\label{conj:MatchingGfunctionModelPeriod}
For any $1\leq n'\leq n$ there exists a constant $C=C_{n,n'}>0$ such that for $\lambda\in(i\RR)^n$ and $\nu\in(i\RR)^{n'}$ with $\lambda_1+\cdots+\lambda_n=\nu_1+\cdots+\nu_{n'}=0$ and $s\in\frac{1}{2}+i\RR$ we have
\begin{enumerate}
\item for $n'=n$:
$$ |G_{\lambda,\nu}(s)| = C\cdot|\ell^{\rm mod}_{\lambda,\nu,s}(\phi_\lambda\otimes\psi_\nu\otimes f_{s-\frac{1}{2}})|. $$
\item for $1\leq n'\leq n-1$:
$$ |G_{\lambda,\nu}(s)| = C\cdot|\ell^{\rm mod}_{\lambda,\nu,s}(\phi_\lambda\otimes\psi_\nu)|. $$
\end{enumerate}
\end{conjecture}

This conjecture will be verified for $n=2$ in Section~\ref{sec:GL2} and for $n=3$ and $n'=1,2$ in Section~\ref{sec:GL3}.

\begin{corollary}
If Conjecture \ref{conj:MatchingGfunctionModelPeriod} holds for a pair $(n,n')$, $1\leq n'\leq n$, then there exists a constant $C=C_{n,n'}>0$ such that for all Maass forms $f$ and $g$ and all $s\in\frac{1}{2}+i\RR$:
$$ |L(s,f\times\overline{g})| = C\cdot|b_{f,\overline{g},s}|. $$
\end{corollary}

\begin{remark}
In a similar situation with $\Gamma_H\subseteq H$ cocompact, Bernstein--Reznikov \cite{BR04} apply \eqref{eq:Proportionality1} resp. \eqref{eq:Proportionality2} to test functions in order to estimate the proportionality constants. This method was also applied in \cite{FS18,MO17}. In our case $\SL(n',\ZZ)$ is not cocompact in $\SL(n',\RR)$, so that this method does not easily generalize. However, we do believe that a more detailed analysis of the geometry of the locally symmetric subspace $\Gamma_H\backslash H/K_H\subseteq\Gamma_G\backslash G/K_G$ does provide a way to modify the ideas of Bernstein--Reznikov.
\end{remark}

\section{Special values of model periods for $\GL(2)$}\label{sec:GL2}

In this section we verify Conjecture~\ref{conj:MatchingGfunctionModelPeriod} for $(n,n')=(2,2)$ and $(n,n')=(2,1)$. By \eqref{eq:ExplicitFormulaSphericalVector} we have:
$$ \phi_\lambda(\overline{n}) = (1+x^2)^{\frac{\lambda_2-\lambda_1-1}{2}}, \qquad \overline{n}=\begin{pmatrix}1&0\\x&1\end{pmatrix}\in\overline{N}_G=\overline{N}_{G,\max}. $$
The formula for $f_r$ is the same with $\lambda=(r,-r)$.

\subsection{Evaluation of the model period for $\GL(2)\times\GL(2)$}

By Lemma~\ref{lem:ModelIntegral1} we have
\begin{align*}
 \ell_{\lambda,\nu,r}^{\mod}(\phi_\lambda,\psi_\nu,f_r) &= \int_{\RR^2}|x-y|^{\lambda_1+\nu_2+r-\frac{1}{2}}(1+x^2)^{\frac{\lambda_2-\lambda_1-1}{2}}(1+y^2)^{-r-\frac{1}{2}}\,d(x,y)\\
 &= 2\int_\RR\int_y^\infty(x-y)^{\lambda_1+\nu_2+r-\frac{1}{2}}(1+x^2)^{\frac{\lambda_2-\lambda_1-1}{2}}(1+y^2)^{-r-\frac{1}{2}}\,dx\,dy
\end{align*}
The integral over $x$ can be evaluated using \eqref{eq:IntFormulaHypergeom1}:
\begin{equation*}
 \frac{\Gamma(\frac{\lambda_1+\nu_2+r+\frac{1}{2}}{2})\Gamma(\frac{-\lambda_2-\nu_2-r+\frac{1}{2}}{2})}{\Gamma(\frac{\lambda_1-\lambda_2+1}{2})}\int_\RR{_2F_1}(\tfrac{-\lambda_2-\nu_2-r+\frac{1}{2}}{2},\tfrac{-\lambda_1-\nu_2-r+\frac{1}{2}}{2};\tfrac{1}{2};-y^2)(1+y^2)^{-r-\frac{1}{2}}\,dy.
\end{equation*}
Substituting $z=y^2$ gives
\begin{equation*}
 \frac{2\Gamma(\frac{\lambda_1+\nu_2+r+\frac{1}{2}}{2})\Gamma(\frac{-\lambda_2-\nu_2-r+\frac{1}{2}}{2})}{\Gamma(\frac{\lambda_1-\lambda_2+1}{2})}\int_0^\infty z^{-\frac{1}{2}}(z+1)^{-r-\frac{1}{2}}{_2F_1}(\tfrac{-\lambda_2-\nu_2-r+\frac{1}{2}}{2},\tfrac{-\lambda_1-\nu_2-r+\frac{1}{2}}{2};\tfrac{1}{2};-z)\,dz.
\end{equation*}
Evaluating with \eqref{eq:IntFormulaHypergeom3} gives
\begin{equation*}
 \frac{2\Gamma(\frac{\lambda_1+\nu_2+r+\frac{1}{2}}{2})\Gamma(\frac{-\lambda_2-\nu_2-r+\frac{1}{2}}{2})\Gamma(\frac{1}{2})\Gamma(\frac{-\lambda_2-\nu_2+r+\frac{1}{2}}{2})\Gamma(\frac{-\lambda_1-\nu_2+r+\frac{1}{2}}{2})}{\Gamma(\frac{\lambda_1-\lambda_2+1}{2})\Gamma(r+\frac{1}{2})\Gamma(\tfrac{-\lambda_1-\lambda_2-2\nu_2+1}{2})}.
\end{equation*}
Since $\lambda_1+\lambda_2+\nu_1+\nu_2=0$ this gives
\begin{equation*}
 \frac{2\Gamma(\frac{\lambda_1+\nu_2+r+\frac{1}{2}}{2})\Gamma(\frac{-\lambda_2-\nu_2-r+\frac{1}{2}}{2})\Gamma(\frac{1}{2})\Gamma(\frac{\lambda_1+\nu_1+r+\frac{1}{2}}{2})\Gamma(\frac{\lambda_2+\nu_1+r+\frac{1}{2}}{2})}{\Gamma(\frac{\lambda_1-\lambda_2+1}{2})\Gamma(r+\frac{1}{2})\Gamma(\tfrac{\nu_1-\nu_2+1}{2})}.
\end{equation*}
We compare this expression to
$$ G_{\lambda,\nu}(s) = \frac{\pi^{\frac{\lambda_1-\lambda_2+\nu_1-\nu_2}{2}}\Gamma(\frac{\lambda_1+\nu_1+s}{2})\Gamma(\frac{\lambda_1+\nu_2+s}{2})\Gamma(\frac{\lambda_2+\nu_1+s}{2})\Gamma(\frac{\lambda_2+\nu_2+s}{2})}{2\pi^{s-1}\Gamma(s)\Gamma(\frac{\lambda_1-\lambda_2+1}{2})\Gamma(\frac{\nu_1-\nu_2+1}{2})} $$
and observe that they agree in absolute value up to a constant for $s=r+\frac{1}{2}\in\frac{1}{2}+i\RR$ and $\lambda,\nu\in(i\RR)^2$.

\subsection{Evaluation of the model period for $\GL(2)\times\GL(1)$}

By Lemma~\ref{lem:ModelIntegral2} we have
\begin{equation*}
 \ell_{\lambda,\nu,s}(\phi_\lambda\otimes\psi_\nu) = \int_\RR |x|^{-\lambda_2-\nu-s}(1+x^2)^{\frac{\lambda_2-\lambda_1-1}{2}}\,dx,
\end{equation*}
which, by \eqref{eq:IntFormulaHypergeom1}, evaluates to
$$ \frac{\Gamma(\frac{-\lambda_2-\nu-s+1}{2})\Gamma(\frac{\lambda_1+\nu+s}{2})}{\Gamma(\frac{\lambda_1-\lambda_2+1}{2})}. $$
We compare this expression to
$$ G_{\lambda,\nu}(s) = \frac{\pi^{\frac{\lambda_1-\lambda_2+1}{2}}\Gamma(\frac{\lambda_1+s}{2})\Gamma(\frac{\lambda_2+s}{2})}{2\pi^s\Gamma(\frac{\lambda_1-\lambda_2+1}{2})} $$
and observe that they agree in absolute value up to a constant for $\lambda\in(i\RR)^2$ and $\nu\in i\RR$ with $\lambda_1+\lambda_2=\nu=0$ and $s\in\frac{1}{2}+i\RR$.

\section{Special values of model periods for $\GL(3)$}\label{sec:GL3}

In this section we verify our Conjecture~\ref{conj:MatchingGfunctionModelPeriod} for $(n,n')=(3,2)$ and $(n,n')=(3,1)$. By \eqref{eq:ExplicitFormulaSphericalVector} we have:
$$ \phi_\lambda(\overline{n}) = (1+y^2+z^2)^{\frac{\lambda_3-\lambda_2-1}{2}}(1+x^2+(z-xy)^2)^{\frac{\lambda_2-\lambda_1-1}{2}}, \qquad \overline{n}=\begin{pmatrix}1&0&0\\x&1&0\\z&y&1\end{pmatrix}\in\overline{N}_G. $$

\subsection{Evaluation of the model period for $\GL(3)\times\GL(2)$}

By Lemma~\ref{lem:ModelIntegral1} we have
\begin{multline*}
 \ell_{\lambda,\nu,s}(\phi_\lambda\otimes\psi_\nu) = \int_{\RR^3}(1+y^2+z^2)^{\frac{\lambda_3-\lambda_2-1}{2}}(1+x^2+(z-xy)^2)^{\frac{\lambda_2-\lambda_1-1}{2}}\\
 \times|z|^{-\lambda_3-\nu_1-s}|x|^{-\lambda_2-\nu_2-s}\,d(x,y,z).
\end{multline*}
Since the integral is invariant under the substitution $(x,y)\mapsto(-x,-y)$, we may change $\int_{\RR^2}\,dx\,dy$ to $2\int_\RR\int_0^\infty\,dx\,dy$. Rewriting $1+x^2+(z-xy)^2=(1+y^2)((x-\frac{yz}{1+y^2})^2+\frac{1+y^2+z^2}{(1+y^2)^2})$ and substituting $x\mapsto x+\frac{yz}{1+y^2}$ yields
\begin{multline*}
 2\int_{\RR^2}\int_{\frac{yz}{1+y^2}}^\infty(1+y^2+z^2)^{\frac{\lambda_3-\lambda_2-1}{2}}(1+y^2)^{\frac{\lambda_2-\lambda_1-1}{2}}(x^2+\tfrac{1+y^2+z^2}{(1+y^2)^2})^{\frac{\lambda_2-\lambda_1-1}{2}}\\
 \times|z|^{-\lambda_3-\nu_1-s}(x-\tfrac{yz}{1+y^2})^{-\lambda_2-\nu_2-s}\,dx\,d(y,z).
\end{multline*}
The integral over $x$ can be computed using \eqref{eq:IntFormulaHypergeom1}:
\begin{multline*}
 \frac{\Gamma(\frac{-\lambda_2-\nu_2-s+1}{2})\Gamma(\frac{\lambda_1+\nu_2+s}{2})}{\Gamma(\frac{\lambda_1-\lambda_2+1}{2})}\int_{\RR^2}(1+y^2+z^2)^{\frac{\lambda_3-\lambda_2-\lambda_1-\nu_2-s-1}{2}}(1+y^2)^{\frac{\lambda_1+\lambda_2+2\nu_2+2s-1}{2}}\\
 \times|z|^{-\lambda_3-\nu_1-s}{_2F_1}(\tfrac{\lambda_1+\nu_2+s}{2},\tfrac{\lambda_2+\nu_2+s}{2};\frac{1}{2};-\tfrac{y^2z^2}{1+y^2+z^2})\,d(y,z).
\end{multline*}
Expanding the hypergeometric function with \eqref{eq:EulerIntegralRepresentation} gives
\begin{multline*}
 \frac{\Gamma(\frac{\lambda_1+\nu_2+s}{2})\Gamma(\frac{1}{2})}{\Gamma(\frac{\lambda_1-\lambda_2+1}{2})\Gamma(\frac{\lambda_2+\nu_2+s}{2})}\int_{\RR^2}\int_0^\infty(1+y^2+z^2)^{\frac{\lambda_3-\lambda_2-1}{2}}(1+y^2)^{\frac{\lambda_1+\lambda_2+2\nu_2+2s-1}{2}}|z|^{-\lambda_3-\nu_1-s}\\
 \times t^{\frac{\lambda_2+\nu_2+s-2}{2}}(1+t)^{\frac{\lambda_1+\nu_2+s-1}{2}}(1+y^2+z^2+t(1+y^2)(1+z^2))^{\frac{-\lambda_1-\nu_2-s}{2}}\,dt\,d(y,z).
\end{multline*}
The integral over $z$ can be computed using \eqref{eq:IntFormula2Quadratics}:
\begin{multline*}
 \frac{\Gamma(\frac{\lambda_1+\nu_2+s}{2})\Gamma(\frac{1}{2})\Gamma(\frac{-\lambda_3-\nu_1-s+1}{2})\Gamma(\frac{\lambda_1+\lambda_2+\nu_1+\nu_2+2s}{2})}{\Gamma(\frac{\lambda_1-\lambda_2+1}{2})\Gamma(\frac{\lambda_2+\nu_2+s}{2})\Gamma(\frac{\lambda_1+\lambda_2-\lambda_3+\nu_2+s+1}{2})}\int_\RR\int_0^\infty(1+y^2)^{\frac{\nu_2-\nu_1-1}{2}}t^{\frac{\lambda_2+\nu_2+s-2}{2}}(1+t)^{-\frac{1}{2}}\\
 \times{_2F_1}(\tfrac{\lambda_1+\nu_2+s}{2},\tfrac{-\lambda_3-\nu_1-s+1}{2};\tfrac{\lambda_1+\lambda_2-\lambda_3+\nu_2+s+1}{2};-\tfrac{t}{1+t}y^2)\,dt\,dy.
\end{multline*}
Substituting $s=\frac{t}{t+1}$ and $sy^2=x$ gives
\begin{multline*}
 \frac{\Gamma(\frac{\lambda_1+\nu_2+s}{2})\Gamma(\frac{1}{2})\Gamma(\frac{-\lambda_3-\nu_1-s+1}{2})\Gamma(\frac{\lambda_1+\lambda_2+\nu_1+\nu_2+2s}{2})}{\Gamma(\frac{\lambda_1-\lambda_2+1}{2})\Gamma(\frac{\lambda_2+\nu_2+s}{2})\Gamma(\frac{\lambda_1+\lambda_2-\lambda_3+\nu_2+s+1}{2})}\int_0^\infty\int_0^1x^{-\frac{1}{2}}(x+s)^{\frac{\nu_2-\nu_1-1}{2}}s^{\frac{\lambda_2+\nu_1+s-2}{2}}\\
 \times(1-s)^{\frac{-\lambda_2-\nu_2-s-1}{2}}{_2F_1}(\tfrac{\lambda_1+\nu_2+s}{2},\tfrac{-\lambda_3-\nu_1-s+1}{2};\tfrac{\lambda_1+\lambda_2-\lambda_3+\nu_2+s+1}{2};-x)\,ds\,dx.
\end{multline*}
The integral over $x$ can be computed using Lemma \ref{lem:IntF21}:
\begin{multline*}
 \frac{\Gamma(\frac{\lambda_1+\nu_2+s}{2})\Gamma(\frac{1}{2})^2\Gamma(\frac{-\lambda_3-\nu_1-s+1}{2})\Gamma(\frac{\lambda_1+\lambda_2+\nu_1+\nu_2+2s}{2})\Gamma(\frac{\nu_1-\nu_2}{2})}{\Gamma(\frac{\lambda_1-\lambda_2+1}{2})\Gamma(\frac{\lambda_2+\nu_2+s}{2})\Gamma(\frac{\lambda_1+\lambda_2-\lambda_3+\nu_2+s+1}{2})\Gamma(\frac{\nu_1-\nu_2+1}{2})}\int_0^1 s^{\frac{\lambda_2+\nu_2+s-2}{2}}(1-s)^{\frac{-\lambda_2-\nu_2-s-1}{2}}\\
 \times{_3F_2}(\tfrac{\lambda_1+\nu_2+s}{2},\tfrac{-\lambda_3-\nu_1-s+1}{2},\tfrac{1}{2};\tfrac{\lambda_1+\lambda_2-\lambda_3+\nu_2+s+1}{2},\tfrac{\nu_2-\nu_1+2}{2};s)\,ds\\
 +\frac{\Gamma(\frac{1}{2})\Gamma(\frac{\lambda_1+\lambda_2+\nu_1+\nu_2+2s}{2})\Gamma(\frac{\lambda_1+\nu_1+s}{2})\Gamma(\frac{-\lambda_3-\nu_2-s+1}{2})\Gamma(\frac{\nu_2-\nu_1}{2})}{\Gamma(\frac{\lambda_1-\lambda_2+1}{2})\Gamma(\frac{\lambda_2+\nu_2+s}{2})\Gamma(\frac{\lambda_1+\lambda_2-\lambda_3+\nu_1+s+1}{2})}\int_0^1 s^{\frac{\lambda_2+\nu_1+s-2}{2}}(1-s)^{\frac{-\lambda_2-\nu_2-s-1}{2}}\\
 \times{_3F_2}(\tfrac{\lambda_1+\nu_1+s}{2},\tfrac{-\lambda_3-\nu_2-s+1}{2},\tfrac{\nu_1-\nu_2+1}{2};\tfrac{\lambda_1+\lambda_2-\lambda_3+\nu_1+s+1}{2},\tfrac{\nu_1-\nu_2+2}{2};s)\,ds.
\end{multline*}
The remaining integrals can be evaluated using \eqref{eq:IntFormulaHypergeometric}:
\begin{multline*}
\frac{\Gamma(\frac{\lambda_1+\nu_2+s}{2})\Gamma(\frac{1}{2})\Gamma(\frac{-\lambda_3-\nu_1-s+1}{2})\Gamma(\frac{\lambda_1+\lambda_2+\nu_1+\nu_2+2s}{2})\Gamma(\frac{\nu_1-\nu_2}{2})\Gamma(\frac{-\lambda_2-\nu_2-s+1}{2})}{\Gamma(\frac{\lambda_1-\lambda_2+1}{2})\Gamma(\frac{\lambda_1+\lambda_2-\lambda_3+\nu_2+s+1}{2})\Gamma(\frac{\nu_1-\nu_2+1}{2})}\\
\times{_4F_3}(\tfrac{\lambda_1+\nu_2+s}{2},\tfrac{-\lambda_3-\nu_1-s+1}{2},\tfrac{1}{2},\tfrac{\lambda_2+\nu_2+s}{2};\tfrac{\lambda_1+\lambda_2-\lambda_3+\nu_2+s+1}{2},\tfrac{\nu_2-\nu_1+2}{2},\tfrac{1}{2};1)\\
+\frac{\Gamma(\frac{1}{2})\Gamma(\frac{\lambda_1+\lambda_2+\nu_1+\nu_2+2s}{2})\Gamma(\frac{\lambda_1+\nu_1+s}{2})\Gamma(\frac{-\lambda_3-\nu_2-s+1}{2})\Gamma(\frac{\nu_2-\nu_1}{2})\Gamma(\frac{\lambda_2+\nu_1+s}{2})\Gamma(\frac{-\lambda_2-\nu_2-s+1}{2})}{\Gamma(\frac{\lambda_1-\lambda_2+1}{2})\Gamma(\frac{\lambda_2+\nu_2+s}{2})\Gamma(\frac{\lambda_1+\lambda_2-\lambda_3+\nu_1+s+1}{2})\Gamma(\frac{\nu_1-\nu_2+1}{2})}\\
\times{_4F_3}(\tfrac{\lambda_1+\nu_1+s}{2},\tfrac{-\lambda_3-\nu_2-s+1}{2},\tfrac{\nu_1-\nu_2+1}{2},\tfrac{\lambda_2+\nu_1+s}{2};\tfrac{\lambda_1+\lambda_2-\lambda_3+\nu_1+s+1}{2},\tfrac{\nu_1-\nu_2+2}{2},\tfrac{\nu_1-\nu_2+1}{2};1).
\end{multline*}
Reducing both ${_4F_3}$'s to ${_3F_2}$'s and applying \eqref{eq:Trafo3F2At1} gives
\begin{multline*}
 \frac{\Gamma(\frac{1}{2})\Gamma(\frac{-\lambda_3-\nu_1-s+1}{2})\Gamma(\frac{\lambda_1+\lambda_2+\nu_1+\nu_2+2s}{2})\Gamma(\frac{-\lambda_2-\nu_2-s+1}{2})\Gamma(\frac{\lambda_2+\nu_1+s}{2})\Gamma(\frac{-\lambda_3-\nu_2-s+1}{2})\Gamma(\frac{\lambda_1+\nu_2+s}{2})}{\Gamma(\frac{\lambda_1-\lambda_2+1}{2})\Gamma(\frac{\nu_1-\nu_2+1}{2})\Gamma(\frac{\lambda_1+\lambda_2-\lambda_3+\nu_2+s+1}{2})\Gamma(\frac{\lambda_2-\lambda_3+1}{2})}\\
 \times{_3F_2}(\tfrac{\lambda_2+\nu_2+s}{2},\tfrac{-\lambda_3-\nu_1-s+1}{2},\tfrac{\lambda_2-\lambda_3+1}{2};\tfrac{\lambda_2-\lambda_3+1}{2},\tfrac{\lambda_1+\lambda_2-\lambda_3+\nu_2+s+1}{2};1).
\end{multline*}
Reducing the ${_3F_2}$ to an ${_2F_1}$ and evaluating it using \eqref{eq:SpecialValueHypergeometric} finally gives
\begin{equation*}
\frac{\Gamma(\frac{1}{2})\Gamma(\frac{-\lambda_3-\nu_1-s+1}{2})\Gamma(\frac{-\lambda_2-\nu_2-s+1}{2})\Gamma(\frac{\lambda_2+\nu_1+s}{2})\Gamma(\frac{-\lambda_3-\nu_2-s+1}{2})\Gamma(\frac{\lambda_1+\nu_2+s}{2})\Gamma(\frac{\lambda_1+\nu_1+s}{2})}{\Gamma(\frac{\lambda_1-\lambda_2+1}{2})\Gamma(\frac{\nu_1-\nu_2+1}{2})\Gamma(\frac{\lambda_2-\lambda_3+1}{2})\Gamma(\frac{\lambda_1-\lambda_3+1}{2})}.
\end{equation*}
We compare this expression to
$$ G_{\lambda,\nu}(s) = \frac{\pi^{\lambda_1-\lambda_3+\frac{\nu_1-\nu_2}{2}+2}\Gamma(\frac{\lambda_1+\nu_1+s}{2})\Gamma(\frac{\lambda_2+\nu_1+s}{2})\Gamma(\frac{\lambda_3+\nu_1+s}{2})\Gamma(\frac{\lambda_1+\nu_2+s}{2})\Gamma(\frac{\lambda_2+\nu_2+s}{2})\Gamma(\frac{\lambda_3+\nu_2+s}{2})}{4\pi^{3s}\Gamma(\frac{\lambda_1-\lambda_2+1}{2})\Gamma(\frac{\lambda_1-\lambda_3+1}{2})\Gamma(\frac{\lambda_2-\lambda_3+1}{2})\Gamma(\frac{\nu_1-\nu_2+1}{2})} $$
and observe that they agree in absolute value up to a constant for $\lambda\in(i\RR)^3$ and $\nu\in(i\RR)^2$ with $\lambda_1+\lambda_2+\lambda_3=\nu_1+\nu_2=0$ and $s\in\frac{1}{2}+i\RR$.

\subsection{Matching of the model period for $\GL(3)\times\GL(1)$}

By Lemma~\ref{lem:ModelIntegral1} we have
\begin{multline*}
 \ell_{\lambda,\nu,s}(\phi_\lambda\otimes\psi_\nu) = \int_{\RR^3}(1+y^2+z^2)^{\frac{\lambda_3-\lambda_2-1}{2}}(1+x^2+(z-xy)^2)^{\frac{\lambda_2-\lambda_1-1}{2}}|y|^{-\lambda_3-\nu-s}\\
 \times e^{-2\pi\sqrt{-1}x}\,d(x,y,z).
\end{multline*}
We first note that the integrand is invariant under the transformation $(x,y,z)\mapsto(x,-y,-z)$, so that we may replace $\int_\RR dy$ by $2\int_0^\infty dy$. We then substitute $z\mapsto z+xy$, write $1+y^2+(z+xy)^2=(1+x^2)[(y+\frac{xz}{1+x^2})^2+\frac{1+x^2+z^2}{(1+x^2)^2}]$ and substitute $y\mapsto y-\frac{xz}{1+x^2}$ to obtain
\begin{multline*}
 = 2\int_{\RR^2}e^{-2\pi\sqrt{-1}x}(1+x^2)^{-\frac{\lambda_2-\lambda_3+1}{2}}(1+x^2+z^2)^{-\frac{\lambda_1-\lambda_2+1}{2}}\\
 \times\int_{\frac{xz}{1+x^2}}^\infty(y-\tfrac{xz}{1+x^2})^{-\lambda_3-\nu-s}(y^2+\tfrac{1+x^2+z^2}{(1+x^2)^2})^{-\frac{\lambda_2-\lambda_3+1}{2}}\,dy\,d(x,z).
\end{multline*}
The inner integral can be evaluated using \eqref{eq:IntFormulaHypergeom1}. Note that the second summand does not contribute to the integral, because it is an odd function of $z$ whereas the remaining terms are even functions of $z$. We therefore obtain
\begin{multline*}
 = \frac{\Gamma(\frac{-\lambda_3-\nu-s+1}{2})\Gamma(\frac{\lambda_2+\nu+s}{2})}{\Gamma(\frac{\lambda_2-\lambda_3+1}{2})}\int_{\RR^2}e^{-2\pi\sqrt{-1}x}(1+x^2)^{\frac{\lambda_2+\lambda_3+2\nu+2s-1}{2}}\\
 \times(1+x^2+z^2)^{-\frac{\lambda_1+\nu+s+1}{2}}{_2F_1}(\tfrac{\lambda_2+\nu+s}{2},\tfrac{\lambda_3+\nu+s}{2};\tfrac{1}{2};-\tfrac{x^2z^2}{1+x^2+z^2})\,d(x,z).
\end{multline*}
By the integral representation \eqref{eq:IntFormulaHypergeom2} for the hypergeometric function this equals
\begin{multline*}
 \frac{\Gamma(\frac{\lambda_2+\nu+s}{2})\Gamma(\frac{1}{2})}{\Gamma(\frac{\lambda_2-\lambda_3+1}{2})\Gamma(\frac{\lambda_3+\nu+s}{2})}\int_{\RR^2}\int_0^1 e^{-2\pi\sqrt{-1}x}(1+x^2)^{\frac{\lambda_2+\lambda_3+2\nu+2s-1}{2}}\\
 \times(1+x^2+z^2)^{-\frac{\lambda_1+\nu+s+1}{2}}t^{\frac{\lambda_3+\nu+s-2}{2}}(1-t)^{\frac{-\lambda_3-\nu-s-1}{2}}(1+\tfrac{x^2z^2}{1+x^2+z^2}t)^{-\frac{\lambda_2+\nu+s}{2}}\,dt\,d(x,z).
\end{multline*}
Rearranging terms this can be written as
\begin{multline*}
 \frac{\Gamma(\frac{\lambda_2+\nu+s}{2})\Gamma(\frac{1}{2})}{\Gamma(\frac{\lambda_2-\lambda_3+1}{2})\Gamma(\frac{\lambda_3+\nu+s}{2})}\int_{\RR^2}\int_0^1 e^{-2\pi\sqrt{-1}x}(1+x^2)^{\frac{-\lambda_1+\lambda_2+\lambda_3+\nu+s-2}{2}}\\
 \times(1+\tfrac{1}{1+x^2}z^2)^{\frac{\lambda_2-\lambda_1-1}{2}}t^{\frac{\lambda_3+\nu+s-2}{2}}(1-t)^{\frac{-\lambda_3-\nu-s-1}{2}}(1+\tfrac{1+tx^2}{1+x^2}z^2)^{-\frac{\lambda_2+\nu+s}{2}}\,dt\,d(x,z).
\end{multline*}
The integral over $z$ can be computed using \eqref{eq:IntFormula2Quadratics}:
\begin{multline*}
 \frac{\Gamma(\frac{\lambda_2+\nu+s}{2})\Gamma(\frac{1}{2})^2\Gamma(\frac{\lambda_1+\nu+s}{2})}{\Gamma(\frac{\lambda_2-\lambda_3+1}{2})\Gamma(\frac{\lambda_3+\nu+s}{2})\Gamma(\frac{\lambda_1+\nu+s+1}{2})}\int_\RR\int_0^1 e^{-2\pi\sqrt{-1}x}(1+x^2)^{\frac{-\lambda_1+\lambda_2+\lambda_3+\nu+s-1}{2}}\\
 \times t^{\frac{\lambda_3+\nu+s-2}{2}}(1-t)^{\frac{-\lambda_3-\nu-s-1}{2}}{_2F_1}(\tfrac{\lambda_2+\nu+s}{2},\tfrac{1}{2};\tfrac{\lambda_1+\nu+s+1}{2};-tx^2)\,dt\,dx.
\end{multline*}
and the integral over $t$ can be computed using \eqref{eq:IntFormulaHypergeometric}:
\begin{multline*}
 \frac{\Gamma(\frac{\lambda_2+\nu+s}{2})\Gamma(\frac{1}{2})\Gamma(\frac{\lambda_1+\nu+s}{2})\Gamma(\frac{-\lambda_3-\nu-s+1}{2})}{\Gamma(\frac{\lambda_2-\lambda_3+1}{2})\Gamma(\frac{\lambda_1+\nu+s+1}{2})}\int_\RR e^{-2\pi\sqrt{-1}x}(1+x^2)^{\frac{-\lambda_1+\lambda_2+\lambda_3+\nu+s-1}{2}}\\
 \times {_3F_2}(\tfrac{\lambda_2+\nu+s}{2},\tfrac{1}{2},\tfrac{\lambda_3+\nu+s}{2};\tfrac{\lambda_1+\nu+s+1}{2},\tfrac{1}{2};-x^2)\,dx.
\end{multline*}
Reducing the hypergeometric function ${_3F_2}$ to ${_2F_1}$, applying the transformation formula \eqref{eq:HypergeomTrafo} and replacing $\int_\RR e^{-2\pi\sqrt{-1}x}\,dx$ by $2\int_0^\infty\cos(2\pi x)\,dx$ we finally get
\begin{multline*}
 \frac{2\Gamma(\frac{\lambda_2+\nu+s}{2})\Gamma(\frac{1}{2})\Gamma(\frac{\lambda_1+\nu+s}{2})\Gamma(\frac{-\lambda_3-\nu-s+1}{2})}{\Gamma(\frac{\lambda_2-\lambda_3+1}{2})\Gamma(\frac{\lambda_1+\nu+s+1}{2})}\\
 \times\int_0^\infty \cos(2\pi x) {_2F_1}(\tfrac{\lambda_1-\lambda_2+1}{2},\tfrac{\lambda_1-\lambda_3+1}{2};\tfrac{\lambda_1+\nu+s+1}{2};-x^2)\,dx.
\end{multline*}
The integral can be evaluated in terms of a $G$-function by \eqref{eq:IntCosHypergeometric}:
\begin{equation*}
 \frac{\Gamma(\frac{\lambda_2+\nu+s}{2})\Gamma(\frac{1}{2})^2\Gamma(\frac{\lambda_1+\nu+s}{2})\Gamma(\frac{-\lambda_3-\nu-s+1}{2})}{\pi\Gamma(\frac{\lambda_2-\lambda_3+1}{2})\Gamma(\frac{\lambda_1-\lambda_2+1}{2})\Gamma(\frac{\lambda_1-\lambda_3+1}{2})}G^{30}_{13}\left(\pi^2\Big|\begin{array}{l}\tfrac{\lambda_1+\nu+s+1}{2}\\\tfrac{1}{2},\tfrac{\lambda_1-\lambda_2+1}{2},\tfrac{\lambda_1-\lambda_3+1}{2}\end{array}\right).
\end{equation*}
Writing the $G$-function as a Mellin--Barnes type integral and shifting the contour shows
\begin{multline*}
 \frac{\pi^{\lambda_1+1}\Gamma(\frac{\lambda_2+\nu+s}{2})\Gamma(\frac{1}{2})^2\Gamma(\frac{\lambda_1+\nu+s}{2})\Gamma(\frac{-\lambda_3-\nu-s+1}{2})}{2\pi\Gamma(\frac{\lambda_2-\lambda_3+1}{2})\Gamma(\frac{\lambda_1-\lambda_2+1}{2})\Gamma(\frac{\lambda_1-\lambda_3+1}{2})}\\
 \times\frac{1}{2\pi\sqrt{-1}}\int_\gamma \frac{\Gamma(\frac{z-\lambda_1}{2})\Gamma(\frac{z-\lambda_2}{2})\Gamma(\frac{z-\lambda_3}{2})}{\Gamma(\frac{z+s+\nu}{2})}\,\pi^{-z}\,dz.
\end{multline*}
We compare this expression to
$$ G_{\lambda,\nu}(s) = \frac{\pi^{\lambda_1-\lambda_3+\frac{3}{2}}\Gamma(\frac{\lambda_1+\nu+s}{2})\Gamma(\frac{\lambda_2+\nu+s}{2})\Gamma(\frac{\lambda_3+\nu+s}{2})}{4\pi^{s}\Gamma(\frac{\lambda_1-\lambda_2+1}{2})\Gamma(\frac{\lambda_1-\lambda_3+1}{2})\Gamma(\frac{\lambda_2-\lambda_3+1}{2})} \frac{1}{2\pi\sqrt{-1}}\int_\gamma \frac{\Gamma(\frac{z-\lambda_1}{2})\Gamma(\frac{z-\lambda_2}{2})\Gamma(\frac{z-\lambda_3}{2})}{\Gamma(\frac{s+z+\nu}{2})}\pi^{-z}\,dz $$
and observe that they agree in absolute value up to a constant for $\lambda\in(i\RR)^3$ and $\nu\in(i\RR)^2$ with $\lambda_1+\lambda_2+\lambda_3=\nu_1+\nu_2=0$ and $s\in\frac{1}{2}+i\RR$.

\appendix

\section{Integral formulas}\label{app:IntFormulas}

We collect some integral formulas for the hypergeometric function and Meijer's $G$-function.

\subsection{Hypergeometric function}

By \cite[equation 7.512~(10)]{GR07} we have for $\Re\gamma,\Re(\alpha-\gamma+\sigma),\Re(\beta-\gamma+\sigma)>0$ and $|\arg z|<\frac{\pi}{2}$:
\begin{multline}
 \int_0^\infty x^{\gamma-1}(x+z)^{-\sigma}{_2F_1}(\alpha,\beta;\gamma;-x)\,dx = \frac{\Gamma(\gamma)\Gamma(\alpha-\gamma+\sigma)\Gamma(\beta-\gamma+\sigma)}{\Gamma(\sigma)\Gamma(\alpha+\beta-\gamma+\sigma)}\\
 \times{_2F_1}(\alpha-\gamma+\sigma,\beta-\gamma+\sigma;\alpha+\beta-\gamma+\sigma;1-z).\label{eq:IntFormulaHypergeom3}
\end{multline}

The following integral formula holds for $|u|>|\beta|$ and $0<\Re\mu<-2\Re\nu$ (see \cite[equation 3.254~(2)]{GR07} for $\lambda=0$):
\begin{multline}
 \int_u^\infty (x-u)^{\mu-1}(x^2+\beta^2)^\nu\,dx = B(\mu,-\mu-2\nu)u^{\mu+2\nu}\\
 \times{_2F_1}\left(-\frac{\mu}{2}-\nu,\frac{1-\mu}{2}-\nu;\frac{1}{2}-\nu;-\tfrac{\beta^2}{u^2}\right).\label{eq:IntRepresentation3F2}
\end{multline}
Using the relation (see \cite[Theorem 2.3.2]{AAR99})
\begin{multline*}
 {_2F_1}(a,b;c;x) = \frac{\Gamma(c)\Gamma(b-a)}{\Gamma(c-a)\Gamma(b)}(-x)^{-a}{_2F_1}(a,a-c+1;a-b+1;x^{-1})\\
 +\frac{\Gamma(c)\Gamma(a-b)}{\Gamma(c-b)\Gamma(a)}(-x)^{-b}{_2F_1}(b,b-c+1;b-a+1;x^{-1})
\end{multline*}
the integral formula in \eqref{eq:IntRepresentation3F2} can be extended to $u\in\RR$, $\beta>0$, by analytic continuation:
\begin{multline}
 \int_u^\infty (x-u)^{\mu-1}(x^2+\beta^2)^\nu\,dx = \frac{\Gamma(\frac{\mu}{2})\Gamma(-\frac{\mu}{2}-\nu)}{2\Gamma(-\nu)}\beta^{\mu+2\nu}{_2F_1}(-\tfrac{\mu}{2}-\nu,\tfrac{1-\mu}{2};\tfrac{1}{2};-\tfrac{u^2}{\beta^2})\\
 -\frac{\Gamma(\frac{\mu+1}{2})\Gamma(\frac{1-\mu}{2}-\nu)}{\Gamma(-\nu)}\beta^{\mu+2\nu-1}u\cdot{_2F_1}(\tfrac{1-\mu}{2}-\nu,\tfrac{2-\mu}{2};\tfrac{3}{2};-\tfrac{u^2}{\beta^2}).\label{eq:IntFormulaHypergeom1}
\end{multline}

For $\Re c>\Re b>0$ the following integral representation holds (see \cite[Theorem 2.2.1]{AAR99}):
\begin{equation}
 {_2F_1}(a,b;c;x) = \frac{\Gamma(c)}{\Gamma(b)\Gamma(c-b)}\int_0^1 t^{b-1}(1-t)^{c-b-1}(1-xt)^{-a}\,dt.\label{eq:IntFormulaHypergeom2}
\end{equation}
The Euler integral representation holds for $\Re(\gamma-\beta),\Re\beta>0$ (see \cite[equation (2.3.17)]{AAR99}):
\begin{equation}
 {_2F_1}(\alpha,\beta;\gamma;1-x) = \frac{\Gamma(\gamma)}{\Gamma(\gamma-\beta)\Gamma(\beta)}\int_0^\infty t^{\beta-1}(1+t)^{\alpha-\gamma}(1+xt)^{-\alpha}\,dt.\label{eq:EulerIntegralRepresentation}
\end{equation}

The following transformation formula holds (see \cite[Theorem 2.2.5]{AAR99}):
\begin{equation}
 {_2F_1}(a,b;c;x) = (1-x)^{c-a-b}{_2F_1}(c-a,c-b;c;x).\label{eq:HypergeomTrafo}
\end{equation}

The following integral formula holds for $\alpha,\beta>0$ and $0<\Re\lambda<2\Re(\mu+\nu)$ (see \cite[3.259~(3)]{GR07}):
\begin{multline}
 \int_0^\infty x^{\lambda-1}(1+\alpha x^2)^{-\mu}(1+\beta x^2)^{-\nu}\,dx\\
 = \frac{1}{2}\alpha^{-\frac{\lambda}{2}}B\left(\frac{\lambda}{2},\mu+\nu-\frac{\lambda}{2}\right){_2F_1}\left(\nu,\frac{\lambda}{2};\mu+\nu;1-\frac{\beta}{\alpha}\right).\label{eq:IntFormula2Quadratics}
\end{multline}

The following integral formula holds for $\Re\mu,\Re\nu>0$ (see \cite[7.512~(12)]{GR07}):
\begin{multline}
 \int_0^1 t^{\mu-1}(1-t)^{\nu-1}{_pF_q}(a_1,\ldots,a_p;b_1,\ldots,b_q;tx)\,dt\\
 =\frac{\Gamma(\mu)\Gamma(\nu)}{\Gamma(\mu+\nu)}{_{p+1}F_{q+1}}(a_1,\ldots,a_p,\mu;b_1,\ldots,b_q,\mu+\nu;x). \label{eq:IntFormulaHypergeometric}
\end{multline}

\begin{lemma}\label{lem:IntF21}
For $\Re\rho,\Re(\alpha-\sigma-\rho+1),\Re(\beta-\sigma-\rho+1)>0$ and $u>0$ we have
\begin{multline*}
 \int_0^\infty x^{\rho-1}(x+u)^{\sigma-1}{_2F_1}(\alpha,\beta;\gamma;-x)\,dx = \frac{\Gamma(\rho)\Gamma(1-\sigma-\rho)}{\Gamma(1-\sigma)}u^{\rho+\sigma-1}{_3F_2}(\alpha,\beta,\rho;\gamma,\sigma+\rho;u)\\
 +\frac{\Gamma(\gamma)\Gamma(\alpha-\sigma-\rho+1)\Gamma(\beta-\sigma-\rho+1)\Gamma(\sigma+\rho-1)}{\Gamma(\beta)\Gamma(\alpha)\Gamma(\gamma-\sigma-\rho+1)}\\
 \times{_3F_2}(\alpha-\sigma-\rho+1,\beta-\sigma-\rho+1,1-\sigma;\gamma-\sigma-\rho+1,2-\sigma-\rho;u).
\end{multline*}
\end{lemma}

Note that the integral in Lemma~\ref{lem:IntF21} is more general than the one in \cite[equation 7.512~(10)]{GR07}, which corresponds to $\rho=\gamma$.

\begin{proof}
We make use of the integral representation (for $\Re c>\Re b>0$, see \cite[Theorem 2.2.1]{AAR99})
$$ {_2F_1}(a,b;c;x) = \frac{\Gamma(c)}{\Gamma(b)\Gamma(c-b)}\int_0^1 t^{b-1}(1-t)^{c-b-1}(1-xt)^{-a}\,dx $$
and the transformation formula (see \cite[Theorem 2.3.2]{AAR99})
\begin{multline*}
 {_2F_1}(a,b;c;x) = \frac{\Gamma(c)\Gamma(c-a-b)}{\Gamma(c-a)\Gamma(c-b)}{_2F_1}(a,b,a+b-c+1;1-x)\\
 + \frac{\Gamma(a+b-c)\Gamma(c)}{\Gamma(a)\Gamma(b)}(1-x)^{c-a-b}{_2F_1}(c-a,c-b;c-a-b+1;1-x).
\end{multline*}
First, substituting $x\mapsto ux$ and using the integral representation we find
\begin{multline*}
 \int_0^\infty x^{\rho-1}(x+u)^{\sigma-1}{_2F_1}(\alpha,\beta;\gamma;-x)\,dx\\
 = \frac{\Gamma(\gamma)}{\Gamma(\gamma-\beta)\Gamma(\beta)}u^{\rho+\sigma-1}\int_0^\infty\int_0^1 t^{\beta-1}(1-t)^{\gamma-\beta-1}x^{\rho-1}(1+x)^{\sigma-1}(1+tux)^{-\alpha}\,dt\,dx.
\end{multline*}
Next, we compute the integral over $x$ with \eqref{eq:EulerIntegralRepresentation}:
$$ = \frac{\Gamma(\gamma)\Gamma(\alpha-\sigma-\rho+1)\Gamma(\rho)}{\Gamma(\gamma-\beta)\Gamma(\beta)\Gamma(\alpha-\sigma+1)}u^{\rho+\sigma-1}\int_0^1 t^{\beta-1}(1-t)^{\gamma-\beta-1}{_2F_1}(\alpha,\rho;\alpha-\sigma+1;1-tu)\,dt. $$
Apply the transformation formula:
\begin{multline*}
 = \frac{\Gamma(\gamma)\Gamma(\rho)\Gamma(1-\sigma-\rho)}{\Gamma(\gamma-\beta)\Gamma(\beta)\Gamma(1-\sigma)}u^{\rho+\sigma-1}\int_0^1 t^{\beta-1}(1-t)^{\gamma-\beta-1}{_2F_1}(\alpha,\rho;\sigma+\rho;tu)\,dt\\
 +\frac{\Gamma(\gamma)\Gamma(\alpha-\sigma-\rho+1)\Gamma(\sigma+\rho-1)}{\Gamma(\gamma-\beta)\Gamma(\beta)\Gamma(\alpha)}\\
 \times\int_0^1 t^{\beta-\sigma-\rho}(1-t)^{\gamma-\beta-1}{_2F_1}(1-\sigma,\alpha-\sigma-\rho+1;2-\sigma-\rho;tu)\,dt
\end{multline*}
and finally \eqref{eq:IntFormulaHypergeometric}:
\begin{multline*}
 = \frac{\Gamma(\rho)\Gamma(1-\sigma-\rho)}{\Gamma(1-\sigma)}u^{\rho+\sigma-1}{_3F_2}(\alpha,\beta,\rho;\gamma,\sigma+\rho;u)\\
 +\frac{\Gamma(\gamma)\Gamma(\alpha-\sigma-\rho+1)\Gamma(\beta-\sigma-\rho+1)\Gamma(\sigma+\rho-1)}{\Gamma(\beta)\Gamma(\alpha)\Gamma(\gamma-\sigma-\rho+1)}\\
 \times{_3F_2}(\alpha-\sigma-\rho+1,\beta-\sigma-\rho+1,1-\sigma;\gamma-\sigma-\rho+1,2-\sigma-\rho;u).\qedhere
\end{multline*}
\end{proof}

For the special value of the generalized hypergeometric function ${_3F_2}$ at $x=1$ we have the following transformation formula, which holds for $\Re(d+e-a-b-c),\Re(c-d+1)>0$ (see \cite[Theorem 2.4.4]{AAR99}):
\begin{multline}
 {_3F_2}(a,b,c;d,e;1) = \frac{\Gamma(d)\Gamma(d-a-b)}{\Gamma(d-a)\Gamma(d-b)}{_3F_2}(a,b,e-c;e,a+b-d+1;1)\\
 +\frac{\Gamma(d)\Gamma(e)\Gamma(d+e-a-b-c)\Gamma(a+b-d)}{\Gamma(a)\Gamma(b)\Gamma(d+e-a-b)\Gamma(e-c)}\\
 \times{_3F_2}(d-a,d-b,d+e-a-b-c;d+e-a-b,d-a-b+1;1)\label{eq:Trafo3F2At1}
\end{multline}

For $\Re(\gamma-\alpha-\beta)>0$ the special value of ${_2F_1}$ at $x=1$ is given by (see \cite[Theorem 2.2.2]{AAR99})
\begin{equation}
 {_2F_1}(\alpha,\beta;\gamma;1) = \frac{\Gamma(\gamma)\Gamma(\gamma-\alpha-\beta)}{\Gamma(\gamma-\alpha)\Gamma(\gamma-\beta)}.\label{eq:SpecialValueHypergeometric}
\end{equation}

\subsection{Meijer's $G$-function}\label{app:GFunction}

For $\Re(a)<1$ and $\Re(b_1),\Re(b_2),\Re(b_3)>0$:
$$ G^{30}_{13}\Big(z\Big|\begin{array}{l}a_1\\b_1,b_2,b_3\end{array}\Big) = \frac{1}{2\pi\sqrt{-1}} \int_\RR \frac{\Gamma(b_1+is)\Gamma(b_2+is)\Gamma(b_3+is)}{\Gamma(a_1+is)}z^{-is} \,ds. $$
We have the following integral representation for $-1<\Re\nu<2\max(\Re\alpha,\Re\beta)-\frac{3}{2}$, $y>0$ (see \cite[8.17~(5)]{EMOT54b}):
\begin{equation}
 \int_0^\infty\cos(xy){_2F_1}(\alpha,\beta;\gamma;-x^2)\,dx = \frac{\Gamma(\frac{1}{2})\Gamma(\gamma)}{\Gamma(\alpha)\Gamma(\beta)}y^{-1}G^{30}_{13}\Big(\big(\tfrac{y}{2}\big)^2\Big|\begin{array}{l}\gamma\\\frac{1}{2},\alpha,\beta\end{array}\Big).\label{eq:IntCosHypergeometric}
\end{equation}

\begin{lemma}\label{lem:IntKBessel}
$$ \int_1^\infty K_\nu(ax)(x^2-1)^\lambda x^\mu\,dx = 2^{-\nu-1}\Gamma(\lambda+1)a^{\nu-1}G^{30}_{13}\Big(\big(\tfrac{a}{2}\big)^2\Big|\begin{array}{l}-\frac{\mu+\nu-2}{2}\\\frac{1}{2},\tfrac{1}{2}-\nu,-\lambda-\tfrac{\mu+\nu}{2}\end{array}\Big). $$
\end{lemma}

\begin{proof}
By \cite[3.771~(2)]{GR07} we have for $a>0$, $\Re x>0$ and $\Re\nu<\frac{1}{2}$
$$ \int_0^\infty(x^2+y^2)^{\nu-\frac{1}{2}}\cos(ay)\,dy = \frac{1}{\sqrt{\pi}}\Big(\frac{2x}{a}\Big)^\nu\cos(\pi\nu)\Gamma(\nu+\frac{1}{2})K_\nu(ax). $$
Multiplying with $(x^2-1)^\lambda x^{\mu-\nu}$ and integrating over $(1,\infty)$ gives
\begin{multline*}
 \int_1^\infty K_\nu(ax)(x^2-1)^\lambda x^\mu\,dx = \frac{2^{-\nu}\sqrt{\pi}}{\cos(\pi\nu)\Gamma(\nu+\frac{1}{2})}a^\nu\\
 \times\int_1^\infty(x^2-1)^\lambda x^{\mu-\nu}\int_0^\infty(x^2+y^2)^{\nu-\frac{1}{2}}\cos(ay)\,dy\,dx.
\end{multline*}
Interchanging the order of integration and substituting $x=t^{-\frac{1}{2}}$ gives
$$ = \frac{2^{-\nu-1}\sqrt{\pi}}{\cos(\pi\nu)\Gamma(\nu+\frac{1}{2})}a^\nu\int_0^\infty\cos(ay)\int_0^1 t^{-\frac{\nu+\mu+2}{2}-\lambda}(1-t)^\lambda(1+ty^2)^{\nu-\frac{1}{2}}\,dt\,dy. $$
The inner integral can be computed in terms of the hypergeometric function using \eqref{eq:IntFormulaHypergeom2}:
$$ = \frac{2^{-\nu-1}\sqrt{\pi}\Gamma(\lambda+1)\Gamma(-\lambda-\frac{\mu+\nu}{2})}{\cos(\pi\nu)\Gamma(\nu+\frac{1}{2})\Gamma(-\frac{\mu+\nu-2}{2})}a^\nu\int_0^\infty\cos(ay){_2F_1}(\tfrac{1}{2}-\nu,-\lambda-\tfrac{\mu+\nu}{2};-\tfrac{\mu+\nu-2}{2};-y^2)\,dy. $$
By \eqref{eq:IntCosHypergeometric} and Euler's reflection formula this equals the claimed formula.
\end{proof}

\providecommand{\bysame}{\leavevmode\hbox to3em{\hrulefill}\thinspace}
\providecommand{\href}[2]{#2}


\begin{thebibliography}{10}
	
	\bibitem{AAR99}
	George~E. Andrews, Richard Askey, and Ranjan Roy, \emph{Special functions},
	Encyclopedia of Mathematics and its Applications, vol.~71, Cambridge
	University Press, Cambridge, 1999.

	\bibitem{BR99}
	Joseph Bernstein and Andre Reznikov, \emph{Analytic continuation of representations and estimates of automorphic forms}, Ann. of Math. (2) \textbf{150} (1999), no.~1, 329--352.
	
	\bibitem{BR04}
	\bysame, \emph{Estimates of automorphic functions},
	Mosc. Math. J. \textbf{4} (2004), no.~1, 19--37.
	
	\bibitem{BR05}
	\bysame, \emph{Periods, subconvexity of {$L$}-functions and representation
		theory}, J. Differential Geom. \textbf{70} (2005), no.~1, 129--141.
	
	\bibitem{BR10}
	\bysame, \emph{Subconvexity bounds for triple {$L$}-functions and
		representation theory}, Ann. of Math. (2) \textbf{172} (2010), no.~3,
	1679--1718.
	
	\bibitem{Blo12}
	Valentin Blomer, \emph{Subconvexity for twisted {$L$}-functions on {${\rm
				GL}(3)$}}, Amer. J. Math. \textbf{134} (2012), no.~5, 1385--1421.
	
	\bibitem{BBM17}
	Valentin Blomer, Jack Buttcane, and P\'{e}ter Maga, \emph{Applications of the
		{K}uznetsov formula on {$\rm GL(3$}) {II}: the level aspect}, Math. Ann.
	\textbf{369} (2017), no.~1-2, 723--759.
	
	\bibitem{Bum88}
	Daniel Bump, \emph{Barnes' second lemma and its application to
		{R}ankin-{S}elberg convolutions}, Amer. J. Math. \textbf{110} (1988), no.~1,
	179--185.
	
	\bibitem{CS15}
	Fulin Chen and Binyong Sun, \emph{Uniqueness of {R}ankin--{S}elberg periods},
	Int. Math. Res. Not. IMRN (2015), no.~14, 5849--5873.
	
	\bibitem{EMOT54b}
	A.~Erd\'elyi, W.~Magnus, F.~Oberhettinger, and F.~G. Tricomi, \emph{Tables of
		integral transforms. {V}ol. {II}}, McGraw-Hill Book Company, Inc., New
	York-Toronto-London, 1954, Based, in part, on notes left by Harry Bateman.
	
	\bibitem{EKP02}
	Pavel Etingof, David Kazhdan, and Alexander Polishchuk, \emph{When is the
		{F}ourier transform of an elementary function elementary?}, Selecta Math.
	(N.S.) \textbf{8} (2002), no.~1, 27--66.
	
	\bibitem{Moe17}
	Jan Frahm, \emph{Symmetry breaking operators for strongly spherical
		reductive pairs}, to appear in Publ. Res. Inst. Math. Sci., available at
	\href{https://arxiv.org/abs/1705.06109}{arXiv:1705.06109}.
	
	\bibitem{FS18}
	Jan Frahm and Feng Su, \emph{Upper bounds for geodesic periods over rank one
		locally symmetric spaces}, Forum Math. \textbf{30} (2018), no.~5, 1065--1077.
	
	\bibitem{Gol06}
	Dorian Goldfeld, \emph{Automorphic forms and {$L$}-functions for the group
		{${\rm GL}(n,\mathbb{R})$}}, Cambridge Studies in Advanced Mathematics, vol.~99,
	Cambridge University Press, Cambridge, 2006, With an appendix by Kevin A.
	Broughan.
	
	\bibitem{GR07}
	I.~S. Gradshteyn and I.~M. Ryzhik, \emph{Table of integrals, series, and
		products}, eighth ed., Elsevier/Academic Press, Amsterdam, 2015.
	
	\bibitem{HM89}
	Jeff Hoffstein and M.~Ram Murty, \emph{{$L$}-series of automorphic forms on
		{${\rm GL}(3,{\bf R})$}}, Th\'{e}orie des nombres ({Q}uebec, {PQ}, 1987), de
	Gruyter, Berlin, 1989, pp.~398--408.
	
	\bibitem{IS13}
	Taku Ishii and Eric Stade, \emph{Archimedean zeta integrals on {${\rm
				GL}_n\times{\rm GL}_m$} and {${\rm SO}_{2n+1}\times{\rm GL}_m$}}, Manuscripta
	Math. \textbf{141} (2013), no.~3-4, 485--536.
	
	\bibitem{Jac72}
	Herv\'{e} Jacquet, \emph{Automorphic forms on {${\rm GL}(2)$}. {P}art {II}},
	Lecture Notes in Mathematics, Vol. 278, Springer-Verlag, Berlin-New York,
	1972.
	
	\bibitem{Kna01}
	Anthony~W. Knapp, \emph{Representation theory of semisimple groups}, Princeton
	Landmarks in Mathematics, Princeton University Press, Princeton, NJ, 2001, An
	overview based on examples, Reprint of the 1986 original.
	
	\bibitem{KS15}
	Toshiyuki Kobayashi and Birgit Speh, \emph{Symmetry breaking for
		representations of rank one orthogonal groups}, Mem. Amer. Math. Soc.
	\textbf{238} (2015), no.~1126.
	
	\bibitem{KS18}
	\bysame, \emph{Symmetry breaking for representations of rank one orthogonal
		groups {II}}, Lecture Notes in Mathematics, vol. 2234, Springer, Singapore,
	2018.
	
	\bibitem{KS19}
	\bysame, \emph{Distinguished representations of {${\rm SO}(n+1,1)\times{\rm
			SO}(n,1)$}, periods and branching laws}, Relative Trace Formulas, Simons Symposia,
			pages 291--319. Springer, 2021.
	
	\bibitem{Li11}
	Xiaoqing Li, \emph{Bounds for {${\rm GL}(3)\times {\rm GL}(2)$} {$L$}-functions
		and {${\rm GL}(3)$} {$L$}-functions}, Ann. of Math. (2) \textbf{173} (2011),
	no.~1, 301--336.
	
	\bibitem{LY02}
	Jianya Liu and Yangbo Ye, \emph{Subconvexity for {R}ankin-{S}elberg
		{$L$}-functions of {M}aass forms}, Geom. Funct. Anal. \textbf{12} (2002),
	no.~6, 1296--1323.
	
	\bibitem{MSY18}
	Mark McKee, Haiwei Sun, and Yangbo Ye, \emph{Improved subconvexity bounds for
		{$GL(2)\times GL(3)$} and {$GL(3)$} {$L$}-functions by weighted stationary
		phase}, Trans. Amer. Math. Soc. \textbf{370} (2018), no.~5, 3745--3769.
	
	\bibitem{MS12}
	Stephen~D. Miller and Wilfried Schmid, \emph{On the rapid decay of cuspidal
		automorphic forms}, Adv. Math. \textbf{231} (2012), no.~2, 940--964.
	
	\bibitem{MO17}
	Jan M\"{o}llers and Bent {\O}rsted, \emph{Estimates for the restriction of
		automorphic forms on hyperbolic manifolds to compact geodesic cycles}, Int.
	Math. Res. Not. IMRN (2017), no.~11, 3209--3236.
	
	\bibitem{Mun15}
	Ritabrata Munshi, \emph{The circle method and bounds for
		{$L$}-functions---{III}: {$t$}-aspect subconvexity for {$GL(3)$}
		{$L$}-functions}, J. Amer. Math. Soc. \textbf{28} (2015), no.~4, 913--938.
	
	\bibitem{Nel20}
	Paul D. Nelson, \emph{Spectral aspect subconvex bounds for $\mathrm{U}_{n+1}\times\mathrm{U}_n$},
		(2020), preprint, available at \href{https://arxiv.org/abs/2012.02187}{arXiv:2012.02187}.
	
	\bibitem{Rez08}
	Andre Reznikov, \emph{Rankin-{S}elberg without unfolding and bounds for
		spherical {F}ourier coefficients of {M}aass forms}, J. Amer. Math. Soc.
	\textbf{21} (2008), no.~2, 439--477.
	
	\bibitem{Sta90}
	Eric Stade, \emph{On explicit integral formulas for {${\rm GL}(n,{\bf
				R})$}-{W}hittaker functions}, Duke Math. J. \textbf{60} (1990), no.~2,
	313--362, With an appendix by Daniel Bump, Solomon Friedberg and Jeffrey
	Hoffstein.
	
	\bibitem{Sta93}
	\bysame, \emph{Hypergeometric series and {E}uler factors at infinity for
		{$L$}-functions on {${\rm GL}(3,\mathbb{R})\times{\rm GL}(3,\mathbb{R})$}}, Amer.
	J. Math. \textbf{115} (1993), no.~2, 371--387.
	
	\bibitem{Sta01}
	\bysame, \emph{Mellin transforms of {${\rm GL}(n,\mathbb{R})$} {W}hittaker
		functions}, Amer. J. Math. \textbf{123} (2001), no.~1, 121--161.
	
	\bibitem{Sta02}
	\bysame, \emph{Archimedean {$L$}-factors on {${\rm GL}(n)\times{\rm GL}(n)$}
		and generalized {B}arnes integrals}, Israel J. Math. \textbf{127} (2002),
	201--219.
	
	\bibitem{SZ12}
	Binyong Sun and Chen-Bo Zhu, \emph{Multiplicity one theorems: the {A}rchimedean
		case}, Ann. of Math. (2) \textbf{175} (2012), no.~1, 23--44.
	
	\bibitem{You11}
	Matthew~P. Young, \emph{The second moment of {${\rm GL}(3)\times{\rm GL}(2)$}
		{$L$}-functions integrated}, Adv. Math. \textbf{226} (2011), no.~4,
	3550--3578.
	
	\bibitem{You13}
	\bysame, \emph{The second moment of {$GL(3)\times GL(2)$} {$L$}-functions at special
		points}, Math. Ann. \textbf{356} (2013), no.~3, 1005--1028.
	
\end{thebibliography}
\end{document}